\documentclass[11pt]{amsart}
\usepackage{amsmath,amssymb,amsfonts,amsthm, amscd,indentfirst}
\usepackage{amsmath,latexsym,amssymb,amsmath,
	amscd,amsthm,amsxtra}
\usepackage{hyperref}\usepackage{url}
\usepackage{color}
\usepackage{pdfpages}
\usepackage{graphics}
\usepackage{enumitem}

\usepackage{graphicx}
\usepackage{caption}
\usepackage{float}
\usepackage{epsfig,here}
\usepackage{subfigure,here}

\newtheorem{definition}{Definition}[section]
\newtheorem{theorem}{Theorem}[section]
\newtheorem{prop}{Proposition}[section]

\newtheorem{corollary}{Corollary}[section]
\newtheorem{lemma}{Lemma}[section]

\newtheorem{remark}{\textbf{Remark}}[section]

\def\rr{\mathbb{R}}

\def\hh{\mathbb{H}}

\def\p{\partial}

\def\th{\theta}

\def\p{\partial}

\def\<{\langle}
\def\>{\rangle}
\def\div{{\rm div}}
\def\n{\nabla}

\def\vp{\varphi}

\def\ep{\epsilon}

\numberwithin{equation} {section}

\begin{document}
	
	\title[Stable $(r+1)$-th capillary hypersurfaces]{Stable $(r+1)$-th capillary hypersurfaces }
	\author{Jinyu Guo}
\address{Department of Mathematical Sciences, Tsinghua University, Beijing, 100084, China}
\email{guojinyu@tsinghua.edu.cn}
	\author{Haizhong Li}
	\address{Department of Mathematical Sciences, Tsinghua University, Beijing, 100084, China}
	\email{lihz@tsinghua.edu.cn}
	\author{Chao Xia}
	\address{School of Mathematical Sciences, Xiamen University, Xiamen, 361005, China}
	\email{chaoxia@xmu.edu.cn}
	\begin{abstract}
In this paper, we propose a new definition of stable $(r+1)$-th capillary hypersurfaces from  variational perspective for any $1\leq r\leq n-1$. More precisely, we define stable $(r+1)$-th capillary hypersurfaces to be smooth local minimizers of a new energy functional under volume-preserving and contact angle-preserving variations. Using the new concept of the stable $(r+1)$-th capillary hypersurfaces, we generalize the stability results of Souam \cite{Souam3} in a Euclidean half-space and Guo-Wang-Xia \cite{GWX3} in a horoball in hyperbolic space for capillary hypersurface to $(r+1)$-th capillary hypersurface case.


	\end{abstract}
	
	
	\maketitle
	
	\medskip
	
	\tableofcontents

	\section{Introduction}
A classical result for constant mean curvature (CMC) hypersurfaces proved by Barbosa-do Carmo \cite{BdC} and Barbosa-do Carmo-Eschenburg \cite{BCE} states that:
	{\it any stable immersed closed CMC hypersurfaces in a space form are geodesic spheres.} Here ``stable'' means the second variation of the area functional is nonnegative for any volume-preserving variations. The following analogous result for stable immersed closed hypersurfaces with constant higher-order mean curvature in space forms has been proved by Alencar-do Carmo-Colares \cite{AdC}, Alencar-do Carmo-Rosenberg \cite{ACR} and  Barbosa-Colares \cite{BCo}.
\begin{theorem}[\cite{AdC, ACR, BCo}]\label{thm0.0}
Let $0\leq r\leq n-1$. An immersed n-dimensional closed constant $(r+1)$-th mean curvature hypersurface in space forms\footnote{We regard an open hemi-sphere as a spherical space form in this paper.} is stable if and only if it is a geodesic sphere.
\end{theorem}
We also mention that B. Palmer \cite{Pa} and the second author with Y. He \cite{HL} proved analogous result for hypersurfaces with constant $(r+1)$-th anisotropic  mean curvature.

The study of capillary hypersurfaces has attracted a lot of attentions in the last decades. In fluid mechanics a capillary surface models an interface between two fluids in the absence of gravity. In fact the free surface of the fluids locally minimizes free energy functional under a volume constraint. We refer to the book of Finn \cite{Finn} for more physical problems about capillary surfaces.
From the geometric variational point of view, a capillary hypersurface in a domain $B$ is a stationary point of free energy functional for volume-preserving variations whose boundary freely moves on $\p B$. By the first variational formula, it is a CMC hypersurface with boundary which intersects $\p B$ at a constant angle.
There are plenty of important works on the existence, regularity and their min-max theory for free boundary or capillary minimal hypersurfaces, see for example \cite{Struwe, GJ, Jost, DM, GJ2, Gruter, LZZ, LZ1, GLWZ} and the references therein.

The study on the classification for stable capillary hypersufaces has been initiated by Ros-Vergasta \cite{RV} for free boundary case and by Ros-Souam \cite{RS} for general capillary case. When $B$ is a Euclidean unit ball, the classification has been recently completed by Nunes \cite{Nu} for free boundary case in two dimensions and eventually by the third author with Wang \cite{WX} for general capillary case in all dimensions, by using a new Minkowski formula involving no boundary term. When $B$ is a Euclidean half-space, the classification has been recently settled by Souam \cite{Souam3}.
\begin{theorem}[\cite{Souam3}]\label{souam-thm}
A compact immersed capillary hypersurface in a Euclidean half-space is stable if and only if it is a spherical cap.
\end{theorem}
The anisotropic version in a half-space has been proved by the first and the third author \cite{GX4}.

Motivated by the concept of higher-order mean curvatures and also the capillary theory, it is natural to ask for a higher-order capillary theory.
In \cite{DE1, DE2}, Damasceno-Elbert introduced a notion of stable capillary hypersurfaces with constant higher-order mean curvature in terms of the associated stability operator, instead of that given by means of a variational problem.


In this paper, we propose a new notion of stability for higher-order capillary hypersurfaces from the variational perspective.
For any $1\leq r\leq n-1$, a $n$-dimensional $(r+1)$-th capillary hypersurface in $B$ is a hypersurface with constant $(r+1)$-th mean curvature $H_{r+1}$ and with boundary intersecting $\p B$ at a constant angle. It is known that the first variation of a higher-order mean curvature integral involves curvature terms in the boundary integral, which violates the capillary boundary condition. To overcome this difficulty,  we make a restriction on the variation class. Precisely, we define a new $(r+1)$-th energy functional $\mathcal{E}_{r+1}$ and show that a $(r+1)$-th capillary hypersurface is a stationary point of $\mathcal{E}_{r+1}$ for any volume-preserving and angle-preserving variations.
We say a $(r+1)$-th capillary hypersurface is stable if the second variation of $\mathcal{E}_{r+1}$ is nonnegative for any volume-preserving and angle-preserving variations.
We emphasize that comparing with the classical capillary theory ($r=0$), we only allow angle-preserving variations in higher-order case. Equivalently speaking, we require that the normal components of variational fields stay in
\begin{eqnarray}\label{stable1-half}
\mathcal{F}=\left\{\varphi\in C^{\infty}(M) \Big| \int_M \varphi\, dA=0\,\,\, \text{and}\,\,\, \nabla_{\mu} \varphi=q\varphi\,\,\, \text{on}\,\,\, \partial M\right\}.
\end{eqnarray}
This is the key point for this new notion of stability (see Proposition \ref{volume-preserving-1} below in detail).


Our first main result in this paper is the following classification for stable $(r+1)$-th capillary hypersurfaces in a $(n+1)$-dimensional Euclidean half-space $\overline{\mathbb{R}^{n+1}_{+}}$.
\begin{theorem}\label{thm0.1-half}
Let $1\leq r\leq n-1$. A compact immersed $(r+1)$-th capillary hypersurface in $\overline{\mathbb{R}^{n+1}_{+}}$ is stable if and only if it is a spherical cap.
\end{theorem}
The proof of Theorem \ref{thm0.1-half} is based on the following higher-order Minkowski-type formula in $\overline{\mathbb{R}^{n+1}_{+}}$:
\begin{equation}\label{aniso-Mink}
\int_{M}[H_{r}(1-\cos\theta\langle E_{n+1},\nu\rangle)-H_{r+1}\langle x,\nu\rangle]\, dA=0, \quad\,\,\,\,\text{for any}\,\,\, 0\leq r\leq n-1.
\end{equation}
Formula \eqref{aniso-Mink} has been proved by Wang, Weng and the third author \cite{WWX1}, which was used to prove Alexandrov-Fenchel inequalities for embedded hypersurfaces with capillary boundary in $\overline{\mathbb{R}^{n+1}_{+}}$. Note that \eqref{aniso-Mink} offers an admissible test function which also satisfies $\nabla_{\mu} \varphi=q\varphi$ on $\p M$.

When $B$ is a $(n+1)$-dimensional horoball in hyperbolic space $\mathbb{H}^{n+1}$, we see that its boundary $\partial B$ is a horosphere, that is, a non-compact totally umbilical hypersurface with all principal curvatures equal to $1$. In the next part, we study a stability problem for $(r+1)$-th capillary hypersurfaces supported on a horosphere. For $r=0$, the classification of stable capillary hypersurfaces supported on a horosphere has been proved by the first and the third authors with Wang in \cite{GWX3}.
\begin{theorem}[\cite{GWX3}]\label{GWXthm}
A compact immersed  capillary hypersurface supported on a horosphere in $\hh^{n+1}$ is stable if and only if it is totally umbilical.
\end{theorem}
 We now establish the following result for $(r+1)$-th capillary hypersurfaces.
 \begin{theorem}\label{thm0.2-horo}
	Let $1\leq r\leq n-1$. A compact immersed $(r+1)$-th capillary hypersurface  supported on a horosphere in $\mathbb{H}^{n+1}$ with at least one elliptic point is stable  if and only if it is totally umbilical and not totally geodesic.
	\end{theorem}
Here the elliptic point means all the principal curvatures at this point are positive. The existence of elliptic point guarantees the ellipticity of operator $L_{r}$ (see Proposition \ref{prop-elliptic}). When $r=0$, $L_{0}=\Delta$ is elliptic automatically.
If the hypersurface intersects a horosphere orthogonally, then there must be an elliptic point. Therefore, we have the following classification for stable free boundary constant $(r+1)$-th mean curvature hypersurfaces.
	\begin{corollary}\label{cor0.1}
		Let $1\leq r\leq n-1$. A compact immersed  free boundary constant $(r+1)$-th mean curvature hypersurface  supported on a horosphere in $\mathbb{H}^{n+1}$ is stable if and only if it is totally umbilical and not totally geodesic.
	\end{corollary}

The proof of Theorem \ref{thm0.2-horo} is based on a higher-order Minkowski-type formula in a horoball in $\hh^{n+1}$ as follows:
\begin{equation}\label{Mink-horo-1}
\int_M [H_{r}\left(V_{n+1}-\cos \theta \bar{g}(x,\nu)\right)-H_{r+1}\bar{g}(X_{n+1},\nu)]\,dA=0,\quad \text{for any}\,\,\, 0\leq r\leq n-1,
\end{equation}
see \cite{GWX3, CP}. 
This formula induces an admissible test function $\varphi_{n+1}\in \mathcal{F}$ (see Proposition \ref{prop-3.1} below). By utilizing the Killing property of position vector field in the hyperbolic space $\hh^{n+1}$, we obtain the desired rigidity result.


\	

This paper is organized as follows. In Section \ref{sec2} we collect some basic properties for elementary symmetric functions.
 In Section \ref{sec3} we calculate the first and second variational formulas of $(r+1)$-th energy functional $\mathcal{E}_{r+1}$ and introduce a new definition of stable $(r+1)$-th capillary hypersurface in space forms for any $1\leq r\leq n-1$.  In Section \ref{sec4} we consider a rigidity of stable $(r+1)$-th capillary hypersurfaces in a Euclidean half-space and then prove Theorem \ref{thm0.1-half} by the higher-order Minkowski-type formula \eqref{aniso-Mink}. In Section \ref{sec5} we focus on the stability of $(r+1)$-th capillary hypersurface supported on a horosphere in hyperbolic space. We prove some useful and powerful geometric formulas for the $(r+1)$-th capillary hypersurfaces supported on a horosphere. By the higher-order Minkowski-type formula \eqref{Mink-horo-1}, we finally construct an admissible test function and prove Theorem \ref{thm0.2-horo} and hence Corollary \ref{cor0.1}.
	
\

\

\section{Preliminaries}\label{sec2}

Let $(\bar M^{n+1}, \bar g)$ be an oriented $(n+1)$-dimensional Riemannian manifold and $B$ be a domain in $\bar M$ with smooth boundary $\p B$ in $\bar M$. Let $x: (M^n, g)\to (\bar M, \bar g)$ be an
isometric immersion of an orientable $n$-dimensional compact manifold $M$ with boundary $\p M$ satisfying $x|_{\p M}: \p M \to \p B$. We say this immersion $x(M)$ is supported on $\partial B$. For the convenience, we do not distinguish $M$ with its image $x(M)$ and $\p M$ with $x(\p M)$ respectively, through all computations are carried out on $M$ by using the pull-back of $x$.


We denote $\bar \n$, $\bar \Delta$ and $\bar \n^2$ by the gradient, the Laplacian and the Hessian on $\bar M$ with respect to $\bar g$ respectively, while $\n$, $\Delta$ and $\n^2$ by the gradient, the Laplacian and the Hessian on $M$ with respect to its induced metric respectively.
We will use the following terminology for four normal vector fields.
We choose one of the unit normal vector field along $x$ and denote it by $\nu$.
We denote $\bar N$ by the unit outward normal to $\p B$ in $B$ and $\mu$ be the unit outward normal to $\p M$ in $M$.
Let $\bar \nu $ be the unit normal to $\p M$ in $\p B$ such that the bases $\{\bar \nu  , \bar N\}$ and $\{\nu, \mu\}$ have the same orientation in the
normal bundle of $\p M\subset \bar M$. See Figure 1, where $B =\rr^{n+1}_{+}$ and $\p B =\rr^{n}$, $n$-dimensional Euclidean space. We denote by $h$ and $h^{\partial B}$  the second fundamental form of $M$ and $\partial B$ in $\bar{M}$ respectively.

\begin{figure}[H]
	\centering
	\includegraphics[height=8cm,width=12cm]{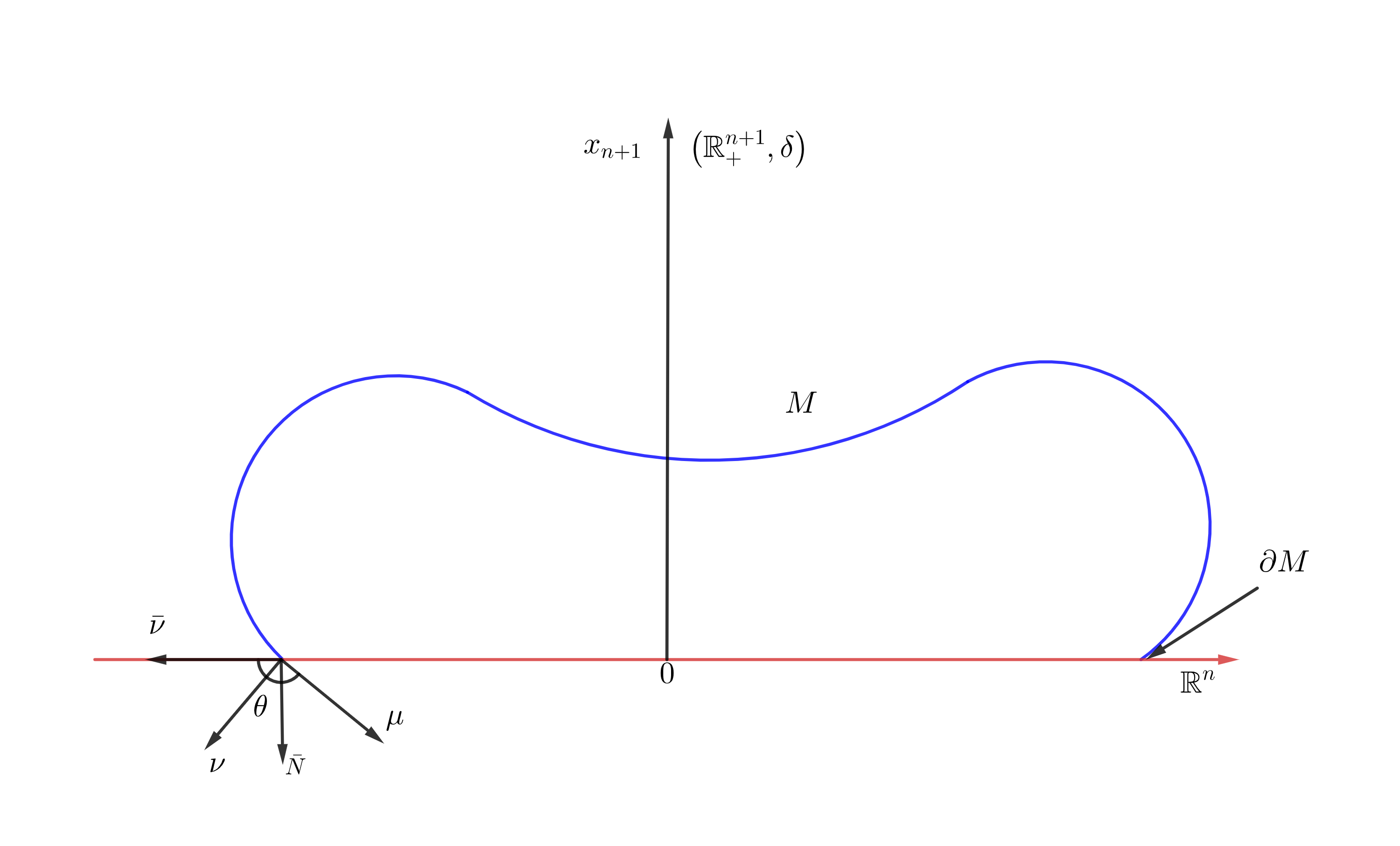}
	\caption{Hypersurface $M$ with contact angle $\theta$ in the half-space $\rr^{n+1}_{+}$.}
\end{figure}

 Under this convention, along $\p M$, the angle between $\mu$ and $\bar \nu$ or equivalently between $\nu$ and $-\bar N$ is equal to $\th$. Precisely, in the normal bundle of $\p M$, we have the following relations:
 \begin{eqnarray}
&&\mu=\sin \th \, \bar N+\cos \th\, \bar \nu ,  \label{mu0}
\\&&\nu=-\cos \th \, \bar N+\sin \th\, \bar \nu  .\label{nu0}
\end{eqnarray}
Equivalently,
 \begin{eqnarray}
&&\bar{N}=\sin \th \, \mu-\cos \th\, \nu,  \label{Nbar}
\\&&\bar \nu =\cos\theta\, \mu+\sin \th\, \nu.\label{nubar}
\end{eqnarray}

Let $\kappa=\left(\kappa_1, \cdots, \kappa_n\right)$ be the vector of principal curvatures of $M$. The $r$-th normalized mean curvature for any $1\leq r\leq n$ is defined by
\begin{equation*}
H_{r}:=\binom{n}{r}^{-1}\sigma_r=\binom{n}{r}^{-1}\sum_{1 \leq i_1<\ldots<i_r \leq n} \kappa_{i_1} \cdots \kappa_{i_r}, \quad \text{where}\,\,\, \binom{n}{r}=\frac{n!}{r!(n-r)!}.
\end{equation*}
It is convenient to set $\sigma_0(\kappa)=1$ and $\sigma_r(\kappa)=0$ for $r>n$.

The Newton tensors are inductively defined to be
\begin{equation}\label{Pr}
 P_{0}=I\quad \text{and} \quad P_{r}=\sigma_{r}I-P_{r-1}h.
\end{equation}
The following is a collection of basic properties about the Newton tensors.
\begin{lemma}[\cite{Guan}]\label{Pr-prop} For any $0\leq r\leq n-1$, we have
\begin{itemize}
 \item [(1)]$P_{r}$ is divergence-free, i.e. $\sum_{j}(P^{ij}_{r})_{,j}=0$;
\item [(2)]${\rm tr}_g(P_{r})=\sum_{i=1}^{n}P_{r}^{ii}=(n-r)\sigma_{r}$;
\item [(3)]${\rm tr}_g(P_{r}h)=\sum_{i,j=1}^{n}P_{r}^{ij}h_{ij}=(r+1)\sigma_{r+1}$;
\item [(4)]${\rm tr}_g (P_rh^{2})=\sum_{i,j,k=1}^{n}P_{r}^{ij}h_{i}^{k}h_{kj}=\sigma_{1}\sigma_{r+1}-(r+2)\sigma_{r+2}$.
\end{itemize}
\end{lemma}
Let $\Gamma_{l}^{+}$ be the G{\aa}rding cone defined by
\begin{equation}\label{Garding}
  \Gamma_{l}^{+}:=\{\kappa \in \mathbb{R}^{n}|\, \sigma_{i}(\kappa)>0,\quad 1 \leq i \leq l\}.
\end{equation}
The Newton-MacLaurin inequalities are as follows:
\begin{lemma}[\cite{Guan}]\label{Newton-Mac} For any $\kappa\in \Gamma_{l}^{+}$, we have
\begin{equation}\label{Newtonaa}
 H_{k-1}(\kappa) H_{l}(\kappa)\leq H_{k}(\kappa) H_{l-1}(\kappa) ,\quad 1 \leq k<l \leq n.
\end{equation}
Equality holds if and only if $\kappa=c(1,\cdots,1)$ for any constant $c>0$.
\end{lemma}
For any $0\leq r\leq n-1$, we denote a second-order operator $L_{r}: C^{\infty}(M)\rightarrow C^{\infty}(M)$ by
\begin{equation}\label{Lr-operator}
  L_{r}f:=\div_{g}(P_{r}\nabla f)=P_{r}\circ\nabla^{2}f,
\end{equation}
where the divergence free property of $P_{r}$ has been used. If $P_{r}$ is positive definite on each point of $M$, then $L_{r}$ is an elliptic operator. In particular, the Laplacian operator is $\Delta =L_{0}$ is elliptic automatically.
On the other hand, we have the following sufficient condition for the ellipticity of $L_r$.
\begin{prop}[{\rm \cite[Proposition 3.2]{BCo}}]\label{prop-elliptic}
 If $H_{r+1}$ is positive and there exists an elliptic point on $M$, then for any $0\leq j\leq r$,
\begin{itemize}
  \item [(i)]each operator $L_{j}=\div_{g}(P_{j}\nabla\cdot)$ is elliptic;
  \item [(ii)]each $j$-th mean curvature $H_{j}$ is positive.
\end{itemize}
\end{prop}

The following proposition is an elementary fact when $M$ is with capillary boundary on $\p B$ and $\p B$ is totally umbilical in $\bar M$.
 \begin{prop}[{\rm \cite[Proposition 2.1]{WX}}]\label{Principle} Assume  $\p B$ is totally umbilical in $\bar M$. Let $x: M\to \bar M$ be an immersion whose boundary $x(\p M)$ intersects $\partial B$ at a constant angle $\theta\in(0,\pi)$. Then $\mu$ is a principal direction of $\p M$ in $M$. Namely, $h(e, \mu)=0$ for any $e\in T(\p M)$. In particular, for any $0\leq r\leq n-1$,
 \begin{equation}\label{principal-direc}
   P_{r}(e,\mu)=0 \quad\forall e\in T(\p M).
 \end{equation}
\end{prop}

\

\section{$(r+1)$-th capillary hypersurfaces and stabilities}\label{sec3}

In this section, we introduce a new notion of stability of $(r+1)$-th capillary hypersurfaces.
Let $\bar M=\mathbb{M}^{n+1}(K)$ be a complete simply-connected $(n+1)$-dimensional Riemannian manifold with constant sectional curvature $K$ and $\partial B$ be a totally umbilical hypersurface in $\mathbb{M}^{n+1}(K)$ with constant principal curvature $\kappa\in\mathbb{R}$ respectively. By a choice of orientation, we can assume $\kappa\in[0,\infty)$. It is a well-known fact that in the Euclidean space and the spherical space-form, geodesic spheres $(\kappa>0)$ and totally geodesic hyperplanes $(\kappa=0)$ are all complete totally umbilical hypersurfaces, while in the hyperbolic space the family of all complete totally umbilical hypersurfaces includes geodesic spheres $(\kappa>1)$, totally geodesic hyperplanes $(\kappa=0)$, horospheres $(\kappa=1)$ and equidistant hypersurfaces $(0<\kappa<1)$ (see e.g. \cite{RAF, RAF3}), among which the horospheres and the equidistant hypersurfaces are non-compact ones.

Let  $x: (-\epsilon, \epsilon)\times M\to \mathbb{M}^{n+1}(K)$ be a differentiable map such that $x(t, \cdot): M\to \mathbb{M}^{n+1}(K)$ is an immersion satisfying $x(t, \p M)\subset\partial B$ for every $t\in (-\ep, \ep)$ and $x(0, \cdot)=x$. $x(t, M)$ is called an admissble variation of $x(0, M)=x(M)$.

The $r$-th area functional  $A_{r}: (-\ep, \ep)\to\rr$ for any $0\leq r\leq n-1$ and the volume functional $V: (-\ep, \ep)\to\rr$ are defined by
\begin{eqnarray*}
&&A_{r}(t)=\int_{M} \sigma_{r}\, dA_t,\\
&&V(t)=\int_{[0,t]\times M} x^*dV,
\end{eqnarray*}
where $dA_t$ is the area element of $M$  with respect to the metric induced by $x(t, \cdot)$ and $dV$ is the volume element of $\mathbb{M}^{n+1}(K)$.
A variation is said to be volume-preserving if $V(t)=V(0)=0$ for each $t\in (-\ep, \ep)$.

Let $Y$ be an admissible variational vector field of $x$ with normal vector field $f\nu$, i.e.
\begin{equation*}
  \frac{\partial x}{\partial t}:=Y=Y^{T}+f\nu,
\end{equation*}
where $Y^{T}$ is tangent to $M$. From \eqref{var-formula} below, we have the first variation formulae of $A_{r}(t)$ and $V(t)$ as follows:
\begin{eqnarray}
A_{r}'(t)&=&(r+1)\int_M \sigma_{r+1}f dA_{t}-K(n-r+1)\int_{M}\sigma_{r-1}fdA_{t}\label{Ar-1}\\
&&+\int_{\partial M}\left(\sigma_{r}\bar{g}(Y,\mu)-\frac{\partial\sigma_{r}}{\partial h_{\mu}^{i}}\nabla_{i} f\right)ds_{t},\nonumber\\
V'(t)&=&\int_M f\, dA_{t}.\label{Vol-1}
\end{eqnarray}
In particular, when $r=0$, we see that
\begin{equation}\label{r=0}
 A_{0}'(t)=\int_{M}\sigma_{1}f\,dA_{t}+\int_{\partial M}\bar{g}(Y,\mu)\, ds_{t}.
\end{equation}
It follows that $A_{0}'(0)=0$ with volume-preserving variation if and only if $M$ is a CMC hypersurface with boundary intersecting $\partial B$ orthogonally, which is exactly a free boundary CMC hypersurface.

However when $r\geq1$, we can not characterize constant $(r+1)$-th mean curvature hypersurface with a constant perpendicular contact angle only by \eqref{Ar-1} and \eqref{Vol-1} directly. Because the integral boundary terms in \eqref{Ar-1} contain $r$-th mean curvature and its derivative. Therefore the key point of this problem is how to define higher-order capillary hypersurface by the variational method reasonably.
In the following, we will give a naturally geometric variational definition for higher-order capillary hypersurface.

We define $r$-th wetting area functional ${W}_{r}: (-\epsilon, \epsilon)\rightarrow \mathbb{R}$ inductively by
\begin{eqnarray}
{W}_{0}(t):=\int_{\partial M\times[0,t]}x^{*}dA_{\partial B},\quad
{W}_{1}(t):=\frac{1}{n}\int_{\partial M}ds_{t},
\end{eqnarray}
and for $2\leq r \leq n-1$,
\begin{equation}\label{wetting-def-sec}
 {W}_{r}(t):=\frac{1}{n}\int_{\partial M}H^{\partial M}_{r-1}ds_{t}+\frac{r-1}{n-r+2}(K+\kappa^{2})W_{r-2}(t),
 \end{equation}
where $dA_{\partial B}$ is the area element of $\partial B$ and $ds_t$ is the area element of $\partial M$ with respect to the metric induced by $x|_{\partial M}(t, \cdot)$ and $H^{\partial M}_{r-1}$ is normalized $(r-1)$-th mean curvature of $\partial M$ in $\partial B$.

Fix a real number $\theta\in (0, \pi)$, we define the $(r+1)$-th energy functional $\mathcal{E}_{r+1}: (-\ep, \ep)\to\rr$ are defined inductively by
\begin{eqnarray*}
\mathcal{E}_{0}(t):=(n+1)V(t),\quad
\mathcal{E}_{1}(t):=A_{0}(t)-\cos \theta W_{0}(t),
\end{eqnarray*}
and for $1 \leq r \leq n-1$,
\begin{equation}\label{Er0099}
  \mathcal{E}_{r+1}(t):=Q_{r+1}(t)+\frac{rK}{n+2-r}\mathcal{E}_{r-1}(t),
\end{equation}
where
\begin{eqnarray}\label{Qr-def-app}
&&Q_{r+1}(t):=\binom{n}{r}^{-1}A_{r}(t)-\cos\theta\sin^{r}\theta W_{r}(t)\\
&&\qquad\qquad-\cos^{r-1}\theta\sum_{l=0}^{r-1}{\frac{(-1)^{r+l}}{n-l}}\kappa^{r-l}\binom{r}{l}\left[(n-r)\cos^{2}\theta+(r-l)\right]\tan^{l}\theta W_{l}(t).\nonumber
\end{eqnarray}
The definition for the higher-order energy functional is motivated by the following first variational formula.
\begin{theorem}\label{thm105-half}
Let $x(\cdot, t): M\to \mathbb{M}^{n+1}(K), t\in (-\epsilon, \epsilon)$ be a family of immersion  supported on $\partial B$ at a constant contact angle $\th \in (0, \pi)$. Assume
\begin{equation}\label{formula106}
  \left(\partial_{t} x\right)^{\perp}=f \nu,
\end{equation}
for  $f\in C^\infty(M)$. Then
\begin{equation}\label{formula107-half}
  \frac{d}{d t} \mathcal{E}_{r+1}(t)=(n-r)\int_{M} H_{r+1} f d A_{t}.
\end{equation}
\end{theorem}
This variational formula has been derived by Wang, Weng and the third author \cite{WeX,WWX1} in the Euclidean half-space and ball case. We postpone the proof of Theorem \ref{thm105-half} to Appendix \ref{app}.

Note that we assume all the immersions $x(\cdot, t)$ in the variational class intersect $\p B$ at a constant angle. We call such variation is an angle-preserving variation.


Now we show an existence theorem for volume-preserving and angle-preserving admissible variations.
Denote a function space as follows:
\begin{eqnarray}\label{stable1-pre}
\mathcal{F}:=\left\{\vp\in C^{\infty}(M)\Big| \int_M \vp\, dA=0\,\,\text{and}\,\,\nabla_{\mu}\varphi=q\varphi\,\,\text{on}\,\, \partial M \right\},
\end{eqnarray}
where
\begin{equation}\label{qsdsdvb}
  q=\kappa\csc\theta+\cot\theta h(\mu,\mu).
\end{equation}

Considering a volume-preserving and angle-preserving admissible variation with variational field having $\vp\nu$ as its normal part. One can see  from \eqref{Vol-1} that $\int_M \varphi\, dA=0$. For the capillary boundary condition $\bar{g}(\nu,\bar N\circ x)=-\cos\theta$, we can get from \eqref{angle-derivative-1} and \eqref{angle-derivative-2} below that
\begin{equation*}\label{derivate-e}
\n_\mu\varphi-q\varphi=-\csc\theta\,\partial_{t}\bar{g}(\nu,\bar{N}\circ x)=0, \,\,\,\,\text{along}\,\, \partial M.
\end{equation*}
Therefore, $\vp\in \mathcal{F}$.

Conversely, we have in the following that $\vp\in \mathcal{F}$ induces a volume-preserving and angle-preserving admissible variation.

\begin{prop}\label{volume-preserving-1} Let $x: M \rightarrow \bar{M}=\mathbb{M}^{n+1}(K)$ be an immersion with boundary $x(\p M)$ intersects $\partial B$ at a constant angle $\theta\in(0,\pi)$.  Then for a given $\varphi \in \mathcal{F}$, there exists an admissible volume-preserving and contact angle-preserving variation of $x$ with variational vector field having $\varphi \nu$ as its normal part.
\end{prop}
\begin{proof}
 We argue as in \cite[Lemma 2.2]{BCE} (also see \cite[Proposition 2.1]{AS}).
We first assume $x: M \rightarrow \bar{M}$ is embedded. For each point $p \in \partial M$, let $\nu_0=\nu+\cos \theta \bar{N}$ be the projection of $\nu$ on $T_{x(p)}(\partial B)$. Denote $\gamma=\frac{1}{\bar{g}\left(\nu, \nu_0\right)} \nu_0-\nu$ which is tangential to $x(M)$ along $\partial M$. We extend $\gamma$ to a smooth vector field on $x(M)$ and still denote by $\gamma$. We let $\eta=\gamma+\nu$ and extend $\eta$ smoothly to a vector field on $U$, which is a $\delta$-neighborhood of $x(M)$ in $\bar{M}$, such that $\eta$ is tangential to $T(\partial B)$ along $\partial B \cap \bar{U}$. By our construction, we see $\bar{g}(\eta, \nu)=1$. Consider the local flow $\zeta_t$ of $\eta$ in $\bar{U}$ satisfying $\frac{\partial}{\partial t} \zeta_t=\eta$. Let $\Theta:\left(-\epsilon, \epsilon\right) \times M \rightarrow \bar{M}$ be given by $\Theta(t, \cdot)=\zeta_t$. We shall find a function $\rho:(-\epsilon, \epsilon) \times M \rightarrow \mathbb{R}$ such that
\begin{equation*}
\tilde{\Theta}(t, \cdot)=\Theta(\rho(t, \cdot), \cdot)
\end{equation*}
is the desired deformation.

First, since $\zeta_t$ is the local flow of $\eta$ and $\eta$ is tangential to $T(\partial B)$ along $\partial B \cap \bar{U}$, we know $\tilde{\Theta}(t, \partial M) \subset \partial B$.
Second, since
\begin{equation*}
\tilde{\Theta}^* d V_{\bar{M}}=\frac{\partial \rho}{\partial t} \Theta^* d V_{\bar{M}}=\frac{\partial \rho}{\partial t} E(\rho(t, \cdot), \cdot) d t d A_{M},
\end{equation*}
where $E(\rho(t, \cdot), \cdot)=\operatorname{det}\left(\left.d \Theta\right|_{(\rho(t, \cdot), \cdot)}\right)$, we have
$$
V(\tilde{\Theta}(t, \cdot))=\int_{[0, t] \times M} \tilde{\Theta}^* d V_{\bar{M}}=\int_{M} \int_0^t \frac{\partial \rho}{\partial t} E(\rho(t, \cdot), \cdot) d t d A_{M} .
$$
Let $\rho(t, \cdot):(-\epsilon, \epsilon) \times M \rightarrow \mathbb{R}$ be the local solution of the following initial value problem:
$$
\frac{\partial \rho}{\partial t}=\frac{\varphi}{E(\rho(t, \cdot), \cdot)}\quad \text{and} \quad \rho(0, \cdot)=0 \,\,\text { in } M.
$$
It follows from the condition $\int_{M} \varphi\, d A=0$ that $V(\tilde{\Theta}(t, \cdot))=0$, that is, $\tilde{\Theta}(t, \cdot)$ is a volume preserving admissible deformation. Now we can easily check that
\begin{equation}\label{sasasqq}
  Y:=\left.\frac{\partial}{\partial t}\right|_{t=0} \tilde{\Theta}(t, \cdot)=\left.\frac{\partial \rho}{\partial t}\right|_{t=0} \cdot \eta(0, \cdot)=\varphi(\gamma+\nu):=Y^{T}+\varphi\nu,
\end{equation}
which means the variational vector field of $\tilde{\Theta}(t, \cdot)$ has $\varphi \nu$ as its normal part.

In immersion case, we shall first construct an admissible variation $\tilde{x}:(-\epsilon, \epsilon) \times$ $M \rightarrow \bar{M}$ and endow $(-\epsilon, \epsilon) \times M$ with the pull-back metric $\tilde{x}^*(\bar{g})$ and it is enough to prove the result for $(-\epsilon, \epsilon) \times M$ endowed with $\tilde{x}^*(\bar{g})$, which is the embedded case.

Finally, since $\tilde{\Theta}(t, \cdot)$ is an admissible variation. From \cite[Appendix]{RS}(also see \eqref{jiut}-\eqref{Ymu4532} in Appendix) along $\partial M$ we have
\begin{equation}\label{Y-1000}
  Y=Y^{T}+\varphi\nu:=Y^{\p M}+\cot\theta\varphi\mu+\varphi\nu=Y^{\p M}+\frac{\varphi}{\sin\theta}\bar{\nu}.
\end{equation}
Here $Y^{\partial M}$ denotes the tangent part of $Y$ to $\partial M$.

By \eqref{mu0}, \eqref{nu0} and the fact that $\partial_{t} \nu=-\nabla \varphi+d\nu\circ Y^{T}$, we have
\begin{eqnarray}\label{angle-derivative-1}
  \partial_{t}\bar{g}(\nu,\bar{N})&=&\bar{g}(\partial_{t}\nu,\bar{N})+\bar{g}(\nu,\partial_{t}\bar{N}) \\
  &=&\bar{g}(\partial_{t}\nu,\sin\theta\mu-\cos\theta\nu)+\bar{g}(\sin\theta\bar{\nu}-\cos\theta\bar{N},\partial_{t}\bar{N}\rangle\nonumber\\
  &=&\sin\theta\bar{g}(\partial_{t}\nu,\mu)+\sin\theta h^{\partial B}(Y,\bar{\nu})\nonumber\\
  &=&\sin\theta(-\nabla_{\mu}\varphi+h(Y^{T},\mu))+h^{\partial B}(Y,\bar{\nu}))\nonumber\\
  &=&\sin\theta\left(-\nabla_{\mu}\varphi+\cot\theta h(\mu,\mu)\varphi+\frac{1}{\sin\theta}h^{\partial B}(\bar{\nu},\bar{\nu})\varphi\right),\nonumber
\end{eqnarray}
where the last equality we used \eqref{Y-1000}, Proposition \ref{Principle} and the fact $\partial B$ is totally umbilical.

By the assumption $\varphi\in\mathcal{F}$, we get
\begin{equation}\label{angle-derivative-2}
\partial_{t}\bar{g}(\nu,\bar{N})=\sin\theta\left(-\nabla_{\mu}\varphi+q\varphi\right)=0.
\end{equation}
Therefore, the boundary contact angle is preserved along the local flow $\zeta_t$.
\end{proof}

Next we define a $(r+1)$-th capillary hypersurface and its stability for any $0\leq r\leq n-1$.
\begin{definition}\label{defcapillary}An immersion is said to be $(r+1)$-th capillary if it is a critical point of the $(r+1)$-th energy functional $\mathcal{E}_{r+1}$ for any volume-preserving and angle-preserving variation of $x$.
\end{definition}

In view of Theorem \ref{thm105-half} we see that a $(r+1)$-th capillary hypersurface is constant $(r+1)$-th mean curvature $H_{r+1}$ and constant contact angle along its boundary. In particular, when the contact angle is $\pi/2$, the hypersurface is called a free boundary constant $(r+1)$-th curvature hypersurface.

\begin{definition}
A $(r+1)$-th capillary hypersurface is called stable if $\mathcal{E}_{r+1}''(0)\ge 0$ for all volume-preserving and angle-preserving  admissible variations.
\end{definition}

For a volume-preserving and angle-preserving admissible variation with variational field having $\vp\nu$ as its normal part, we see from Proposition \ref{formula111} in Appendix (also see \cite[Proposition 4.1]{BCo}) that
 \begin{equation}\label{sigeman}
  \partial_{t}\sigma_{r+1}=-L_{r}\varphi-tr(P_{r}h^2)\varphi-Ktr(P_{r})\varphi+\nabla_{(\partial x/\partial t)^{T}}\sigma_{r+1}.
\end{equation}
Here $P_{r}$ is defined by \eqref{Pr}. Thus the second variational formula of $\mathcal{E}_{r+1}$ is given by
\begin{eqnarray}\label{stab-ineq}
\mathcal{E}_{r+1}''(0)&=&(n-r)\binom{n}{r+1}^{-1}\left[\int_M\sigma_{r+1}'\varphi dA+\int_M\sigma_{r+1}\frac{d}{dt}\Big|_{t=0}(\varphi dA_{t})\right]\\
&=&(n-r)\binom{n}{r+1}^{-1}\left[\int_M\sigma_{r+1}'\varphi dA+\sigma_{r+1}V''(0)\right]\nonumber\\
&=&-(n-r)\binom{n}{r+1}^{-1}\int_M\vp[L_{r}\varphi+tr(P_{r}h^2)\varphi+Ktr(P_{r})\vp] dA,\nonumber
\end{eqnarray}
where we used $\sigma_{r+1}$ is constant and $x$ is volume-preserving.

From the above calculation we have the following proposition.
\begin{prop}
 A $(r+1)$-th capillary hypersurface is stable if and only if \begin{eqnarray}\label{stable1-half}
-\int_M\vp[L_{r}\varphi+tr(P_{r}h^2)\varphi+Ktr(P_{r})\vp] dA\ge 0, \quad \forall \vp\in\mathcal{F},
\end{eqnarray}
where $\mathcal{F}$ is the functional space given by \eqref{stable1-pre}.
\end{prop}

\begin{remark}
We have no boundary integral term in the the second variation formula because we restrict to  the angle-preserving variations.
For the classical capillary theory, that is $r=0$,  the second variation formula involves boundary integral because there is no such restriction (see e.g. \cite{Ba, CK, LiXiong, LiXiong2, WXsurvey}).
\end{remark}

\

\section{Rigidity for stable $(r+1)$-th capillary hypersurfaces in a half-space}\label{sec4}
In this section we consider the case  $B$ is a Euclidean half-space $\overline{\rr^{n+1}_{+}}$, where
 \begin{eqnarray}\label{half-space555}
\rr^{n+1}_{+}=\{x=(x_{1}, x_{2},\cdots,x_{n+1})\in \rr^{n+1}: x_{n+1}>0\}.
\end{eqnarray}
\begin{prop}\label{stable-necess-half} Any spherical caps in the half-space $\overline{\mathbb{R}^{n+1}_{+}}$ are stable $(r+1)$-th capillary hypersurfaces.
\end{prop}
\begin{proof}
Let $\Sigma$ be a spherical cap in $\overline{\mathbb{R}^{n+1}_{+}}$ whose principal curvatures are all equal to $\Lambda>0$. Then
\begin{equation*}
  \sigma_{r}=\binom{n}{r}\Lambda^{r}.
\end{equation*}
It follows from Lemma \ref{Pr-prop} that
\begin{equation*}
L_{r}\varphi=\binom{n-1}{r}\Lambda^{r}\Delta \varphi
\end{equation*}
and
\begin{equation*}
tr(P_{r}h^{2})=\sigma_{1}\sigma_{r+1}-(r+2)\sigma_{r+2}=n\binom{n-1}{r}\Lambda^{r+2}.
\end{equation*}
Hence from \eqref{stab-ineq}, for any $\varphi\in\mathcal{F}$ we have
\begin{eqnarray}\label{staqq}
\mathcal{E}_{r+1}''(0)&=&-(n-r)\binom{n}{r+1}^{-1}\int_\Sigma(L_{r}\varphi+tr(P_{r}h^2)\varphi)\vp\, dA\\
&=&-\frac{(r+1)(n-r)}{n}\Lambda^{r}\int_\Sigma(\Delta\varphi+n\Lambda^{2}\varphi)\vp\, dA\nonumber\\
&=&\frac{(r+1)(n-r)}{n}\Lambda^{r}\mathcal{E}_{1}''(0).\nonumber
\end{eqnarray}
Recall that a spherical cap in $\overline{\mathbb{R}^{n+1}_{+}}$ is a stable capillary hypersurface and minimize the energy functional $\mathcal{E}_{1}$ (see \cite{GMT}). Therefore, $\mathcal{E}_{1}''(0)\ge 0$. It follows that $\mathcal{E}_{r+1}''(0)\ge 0$ and $\Sigma$ is stable $(r+1)$-th capillary hypersurface. The proof is finished.
\end{proof}

We need the following higher-order Minkowski-type formula in $\overline{\rr^{n+1}_{+}}$ from \cite{WWX1}.
\begin{prop}[{\rm \cite[Proposition 2.6]{WWX1}}] \label{Minkhorol}  Let $x: M\to \overline{\mathbb{R}^{n+1}_{+}}$ be an immersed hypersurface  whose boundary intersects $\partial\mathbb{R}^{n+1}_{+}$ at a constant contact angle $\th \in (0, \pi)$. Then
\begin{equation}\label{Mink1-half}
  \int_M[(1-\cos \theta \langle E_{n+1},\nu\rangle)H_{r}-\langle x,\nu\rangle H_{r+1}]\,dA=0, \quad \text{for any}\,\,\,0\leq r\leq n-1,
\end{equation}
where $E_{n+1}=(0,\cdots,0,1)\in\overline{\rr^{n+1}_{+}}$ and $x$ is position vector field in $\overline{\rr^{n+1}_{+}}$.
\end{prop}

Next we derive the equations for several geometric quantities. We denote the $r$-th Jacobi operator by
\begin{equation}\label{Irf11}
  J_{r}f:=L_{r}f+tr(P_{r}h^{2})f, \quad \text{for any}\,\,0\leq r\leq n-1.
\end{equation}
\begin{prop}\label{prop4.66-half} Let $x: M\to\overline{\mathbb{R}^{n+1}_{+}}$ be an immersed hypersurface. Then
the following identities hold on $M$
	\begin{eqnarray}
	J_{r}(\langle E_{n+1},\nu\rangle)&=&\langle E_{n+1},\nabla\sigma_{r+1}\rangle,\label{Ennu}\\
	J_{r}(\langle x,\nu\rangle)&=&\langle x,\nabla \sigma_{r+1}\rangle+(r+1)\sigma_{r+1}.\label{Xmu1}
	\end{eqnarray}
\end{prop}
\begin{proof}
The above formulas have been shown in \cite{BCo}. For the convenience of reader, we give a direct computation.

For a fixed point $p\in M$, let $\{e_{i}\}_{i=1}^{n}$ at $p$ be the local orthonormal basis and $\nabla_{e_{i}}e_{j}|_{p}=0$.
In the following we calculate at $p$. We have
\begin{eqnarray*}
J_{r}(\langle \nu, E_{n+1}\rangle)&=&L_{r}(\langle\nu, E_{n+1}\rangle)+\langle \nu, E_{n+1}\rangle tr(P_{r}h^{2})\\
&=&(P_{r}^{ij}h_{jk}\langle e_{k},E_{n+1}\rangle)_{,i}+\langle\nu,E_{n+1}\rangle P_{r}^{ij}h^{2}_{ij}\\
&=&[(\sigma_{r+1}\delta_{ik}-P_{r+1}^{ik})\langle e_{k},E_{n+1}\rangle]_{,i}+\langle\nu,E_{n+1}\rangle(\sigma_{r+1}\delta_{ik}-P_{r+1}^{ik})h_{ik}\\
&=&\sigma_{r+1,i}\langle e_{i},E_{n+1}\rangle+\sigma_{r+1}(\langle \bar{\nabla}_{e_{i}}e_{i}, E_{n+1}\rangle+\langle \nu, E_{n+1}\rangle tr(h))\\
&&-(P_{r+1}^{ik}\langle \bar{\nabla}_{e_{i}}e_{k},E_{n+1}\rangle+\langle\nu,E_{n+1}\rangle tr(P_{r+1}h))\\
&=&\langle\nabla\sigma_{r+1},E_{n+1}\rangle,
\end{eqnarray*}
where the third equality we used the relation $P_{r+1}=\sigma_{r+1}I-P_{r}\circ h$ and Lemma \ref{Pr-prop}.
\begin{eqnarray*}
J_{r}(\langle x,\nu\rangle)&=&L_{r}(\langle x,\nu\rangle)+\langle x,\nu\rangle tr(P_{r}h^{2})\\
&=&(P_{r}^{ij}h_{jk}\langle x,e_{k}\rangle)_{,i}+\langle x,\nu\rangle P_{r}^{ij}h^{2}_{ij}\\
&=&[(\sigma_{r+1}\delta_{ik}-P_{r+1}^{ik})\langle x,e_{k}\rangle]_{,i}+\langle x,\nu\rangle(\sigma_{r+1}\delta_{ik}-P_{r+1}^{ik})h_{ik}\\
&=&\sigma_{r+1,i}\langle x,e_{i}\rangle+\sigma_{r+1}[(\langle x,e_{i}\rangle)_{,i}+\langle x,\nu\rangle tr(h)]-[P_{r+1}^{ik}(\langle x,e_{k}\rangle)_{,i}+\langle x,\nu\rangle tr(P_{r+1}h)].
\end{eqnarray*}
Since
\begin{eqnarray*}
P_{r+1}^{ik}(\langle x,e_{k}\rangle)_{i}+\langle x,\nu\rangle tr(P_{r+1}h)
&=&P_{r+1}^{ik}(\delta_{ik}-h_{ik}\langle x,\nu\rangle)+\langle x,\nu\rangle tr(P_{r+1}h)\\
&=&tr(P_{r+1})=(n-r-1)\sigma_{r+1}.
\end{eqnarray*}
Therefore, we see that
\begin{equation*}\label{dsdsds}
J_{r}(\langle x,\nu\rangle)=\langle x,\nabla\sigma_{r+1}\rangle+n\sigma_{r+1}-(n-r-1)\sigma_{r+1}=\langle x,\nabla\sigma_{r+1}\rangle+(r+1)\sigma_{r+1}.
\end{equation*}
\end{proof}
Next we check the boundary equations of corresponding geometric quantities.
\begin{prop}[\cite{GX4}]\label{prop-boundary-half} Let $x: M\to  \overline{\rr^{n+1}_{+}}$ be an isometric immersion. Assume $M$ intersects $\partial\mathbb{R}^{n+1}_{+}$ at a constant contact angle $\th \in (0, \pi)$. Then along $\p M$, we have
\begin{eqnarray}\label{Hyp-dddd1}
&&\nabla_{\mu}\langle x,\nu\rangle=q\langle x,\nu\rangle,\label{boundary-222}\\
&&\nabla_{\mu}(1-\cos\theta\langle E_{n+1},\nu\rangle)=q(1-\cos\theta\langle E_{n+1},\nu\rangle),\label{boundary-111}
\end{eqnarray}
where $q=\cot\theta h(\mu,\mu)$.
\end{prop}
\begin{proof}
By choosing $F\equiv1$ in \cite[Proposition 2.2]{GX4}, we get this proposition.
\end{proof}

For the ease of notation, we denote
\begin{equation}\label{uuuuu1-half}
\omega:=1-\cos\theta\langle E_{n+1},\nu\rangle.
\end{equation}
Since $\th \in (0, \pi)$, thus $\omega>0$ on $M$. Motivated by the higher-order Minkowski-type formula \eqref{Mink1-half},
we let
\begin{equation}\label{test-function-half}
\varphi:=\alpha \omega-H_{r+1}\langle x,\nu\rangle,
\end{equation}
where $$\alpha=\left(\int_{M}\omega dA\right)^{-1}\int_{M}\omega H_{r}\,dA.$$
\begin{prop}\label{prop-3.1}Let $x: M\to\overline{\mathbb{R}^{n+1}_{+}}$ be a compact immersed $(r+1)$-th capillary hypersurface with a contact angle $\theta\in(0,\pi)$. 
Then there holds:
	\begin{eqnarray}
J_{r}\varphi&=&\binom{n}{r+1}\left[\alpha(nH_{1}H_{r+1}-(n-r-1)H_{r+2})-(r+1)H_{r+1}^{2}\right],\label{varphi1-half}\\
\nabla_{\mu}\varphi&=&q\varphi,\label{Hyp-bdy1-half}\\
\int_{M}\varphi\,dA&=&0.\label{bdy-zero1-half}
\end{eqnarray}
\end{prop}
\begin{proof}
	\eqref{varphi1-half} follows from Proposition \ref{prop4.66-half} and \eqref{Hyp-bdy1-half} from Proposition \ref{prop-boundary-half}.
	\eqref{bdy-zero1-half} is from \eqref{test-function-half} and \eqref{Mink1-half}.
\end{proof}
Now we prove a rigidity theorem for stable $(r+1)$-th capillary hypersurface in a half-space.
\begin{theorem}\label{thmm4.1}
	Let $M$ be a compact immersed $(r+1)$-th capillary hypersurface in $\overline{\mathbb{R}^{n+1}_{+}}$ with a constant contact angle $\theta\in(0,\pi)$. If $M$ is stable, then it is a spherical cap.
\end{theorem}
\begin{proof}
The proof is close to \cite{HL} by the second author with Y. He for closed hypersurface case. We prove it here for reader's convenience. Since $M$ is compact hypersurface in $\overline{\mathbb{R}^{n+1}_{+}}$ with contact angle $\theta\in(0,\pi)$, there exists a point in $M$ where all the principal curvatures are positive. Thus $H_{r+1}$ is positive constant. From Proposition \ref{prop-elliptic}, we know that each operator $L_{i}$ is elliptic and $H_{i}>0$ for any $i\in\{1,\cdots,r\}$. Hence $\alpha$ is positive constant.

From \eqref{Hyp-bdy1-half} and \eqref{bdy-zero1-half}, we obtain that $\varphi$ is an admissible test function in \eqref{stable1-half}.
Therefore, by \eqref{varphi1-half}, we have
\begin{eqnarray}\label{xeq1111-half}
0 &\le &-\int_M\varphi J_{r}\varphi\, dA\\
&=&-\binom{n}{r+1}\int_M \varphi[\alpha(nH_{1}H_{r+1}-(n-r-1)H_{r+2})-(r+1)H_{r+1}^{2}] dA\nonumber\\
&=&-\alpha \binom{n}{r+1}\int_M\varphi(nH_{1}H_{r+1}-(n-r-1)H_{r+2})dA,\nonumber
\end{eqnarray}
where the last equality we used $H_{r+1}$ is constant and \eqref{bdy-zero1-half}. From \eqref{test-function-half} we see that
\begin{eqnarray}\label{xeq11115555-half}
&&\varphi(nH_{1}H_{r+1}-(n-r-1)H_{r+2})\\
&=&\alpha \omega(nH_{1}H_{r+1}-(n-r-1)H_{r+2})\nonumber\\
&&+(H_{r+1}\langle x,\nu\rangle)(nH_{1}H_{r+1}-(n-r-1)H_{r+2}) \nonumber\\
&=&\alpha(n-r-1)\omega(H_{1}H_{r+1}-H_{r+2})+\alpha (r+1)\omega H_{1}H_{r+1}\nonumber\\
&&+(H_{r+1}\langle x,\nu\rangle)(nH_{1}H_{r+1}-(n-r-1)H_{r+2}). \nonumber
\end{eqnarray}
Putting \eqref{xeq11115555-half} into \eqref{xeq1111-half}, we get
\begin{eqnarray}\label{xeytf}
\qquad0 &\geq&\int_M\varphi(nH_{1}H_{r+1}-(n-r-1)H_{r+2})dA\\
&=&\alpha(n-r-1)\int_M \omega(H_{1}H_{r+1}-H_{r+2})dA+\alpha(r+1)\int_M \omega H_{1}H_{r+1}dA \nonumber\\
&&+\int_M(H_{r+1}\langle x,\nu\rangle)(nH_{1}H_{r+1}-(n-r-1)H_{r+2})dA. \nonumber
\end{eqnarray}
From Lemma \ref{Newton-Mac}, we have
\begin{eqnarray}\label{xeqqqq1}
\qquad0 &\geq&\alpha(r+1)\int_M \omega H_{1}H_{r+1}dA+\int_M(H_{r+1}\langle x,\nu\rangle)(nH_{1}H_{r+1}-(n-r-1)H_{r+2})dA \\
&=&\alpha(r+1)\int_M \omega H_{1}H_{r+1}dA+(r+1)H^{2}_{r+1}\int_M \omega dA\nonumber\\
&=&\alpha(r+1)H_{r+1}^{2}\left[\frac{1}{H_{r+1}}\int_M \omega H_{1}dA-\frac{1}{\alpha}\int_M \omega dA\right].\nonumber
\end{eqnarray}
Here the second equality we used the higher-order Minkowski-type formula \eqref{Mink1-half} twice.

By H\"{o}lder inequality and Newton-MacLaurin inequality \eqref{Newtonaa}, we obtain
\begin{eqnarray}\label{xeqqqq1qqa}
\left(\int_{M}\omega dA\right)^{2}\leq\int_{M}\frac{1}{H_{1}}\omega dA\int_{M}\omega{H_{1}}dA\leq\int_{M}\frac{H_{r}}{H_{r+1}}\omega dA\int_{M}\omega{H_{1}}dA.
\end{eqnarray}
Hence, from \eqref{xeqqqq1qqa} and the definition of $\alpha$ we have
\begin{eqnarray}\label{xewewe}
&&\frac{1}{H_{r+1}}\int_M \omega H_{1}dA-\frac{1}{\alpha}\int_M \omega dA\\
&&=\frac{1}{H_{r+1}}\left[\int_M \omega H_{1}dA-\left(\int_{M}\frac{H_{r}}{H_{r+1}}\omega dA\right)^{-1}\left(\int_{M}\omega dA\right)^{2}\right]\geq0.\nonumber
\end{eqnarray}
Combining \eqref{xeqqqq1} and \eqref{xewewe}, we see the above inequality is in fact an equality. It follows that $H_{r}=H_{r+1}H_{1}$ on $M$ and in turn $M$ is totally umbilical in $\overline{\mathbb{R}^{n+1}_{+}}$, i.e. $M$ is a spherical cap.

\end{proof}

\

\section{$(r+1)$-th capillary hypersurfaces supported on a horosphere}\label{sec5}
In this section we focus on the stability of $(r+1)$-capillary hypersurfaces with boundary supported on a horosphere in hyperbolic space.

Let $(\mathbb{H}^{n+1},\bar{g})$ be a complete simply-connected Riemannian manifold with constant sectional curvature $-1$. We use the upper half-space model for $\hh^{n+1}$, which is denoted by
\begin{eqnarray}\label{half-space}
\hh^{n+1}=\{x=(x_{1}, x_{2},\cdots,x_{n+1})\in \rr^{n+1}_+: x_{n+1}>0\},\quad \bar g=\frac{1}{x_{n+1}^2}\delta,
\end{eqnarray}
where $\delta$ is a Euclidean metric in $\rr^{n+1}$.

A horosphere, a ``sphere" in $\hh^{n+1}$ whose centre lies at $\partial_{\infty}\mathbb{H}^{n+1}$, up to a hyperbolic isometry,  is written by the horizontal plane
\begin{equation}
\mathcal{H}=\{x\in\mathbb{R}^{n+1}_{+}:x_{n+1}=1\}.
\end{equation}
We choose $\bar N=-E_{n+1}=(0,\cdots, 0, -1)$, then all principal curvatures of a horosphere are $\kappa=1$. By Gauss equation, namely $R_{ijij}=-1+\kappa_{i}\kappa_{j}=0$ for any $i\neq j$, we know that a horosphere is isometric to the $n$-dimensional Euclidean space $\rr^{n}$.

Let $x: (M^n, g)\to (\mathbb{H}^{n+1}, \bar g)$ be an isometric immersion of an orientable $n$-dimensional compact manifold $M$ with boundary $\p M$ satisfying $x|_{\p M}: \p M \to \mathcal{H}$. Such an immersion is called an immersion supported on a horosphere $\mathcal{H}$. (See Figure 2 below).

\begin{figure}[H]
	\centering
	\includegraphics[height=6cm,width=15cm]{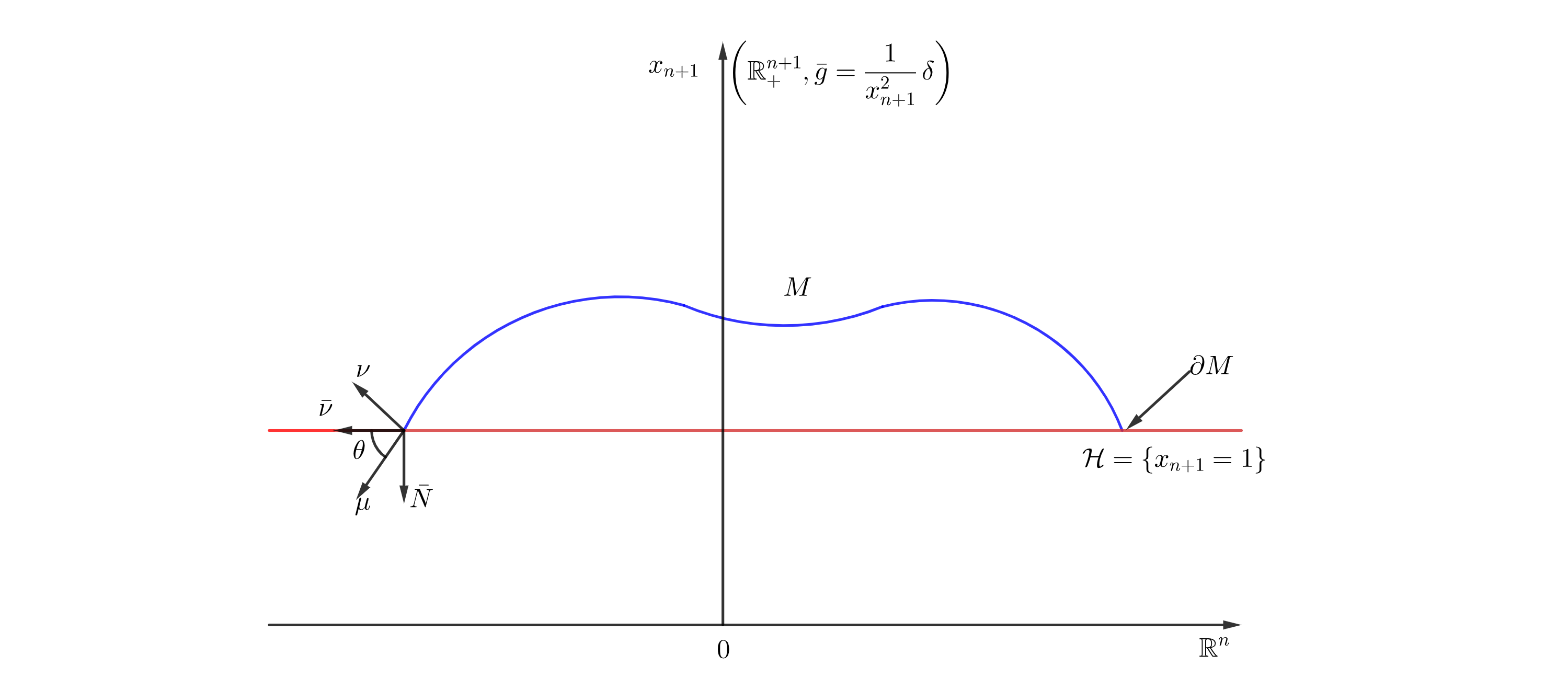}
	\caption{Hypersurface $M$  supported on horosphere $\mathcal{H}$.}
\end{figure}

We denote $x$ to be the position vector in $\hh^{n+1}$ and $\bar{\nabla}$ be the Levi-Civita connection of $\hh^{n+1}$.
We use $\<\cdot ,\cdot\>$ and $\bar g$ to denote the inner product of $\rr^{n+1}$ and $\hh^{n+1}$ respectively,  $D$ and $\bar \nabla$ to denote the Levi-Civita connection of $\rr^{n+1}$ and $\hh^{n+1}$ respectively. Let $\{E_A\}_{A=1}^{n+1}$ be the canonical basis of $\rr^{n+1}$ and $\bar E_A = x_{n+1} E_A$. Then $\{ \bar E_A\}_{A=1}^{n+1}$  is an orthonormal basis of $\hh^{n+1}$ with respect to $\bar{g}$.

The relationship of $\bar{\nabla}$ and $D$ is given by
\begin{equation}\label{YZ}
  \bar{\nabla}_{Y}Z=D_{Y}Z-Y(\ln x_{n+1})Z-Z(\ln x_{n+1})Y+\langle Y,Z\rangle D(\ln x_{n+1}).
\end{equation}
It is easy to check that
\begin{eqnarray}
\bar \n_Y x &=& -\bar g (Y, \bar E_{n+1})x + \bar g (Y, x ) \bar E_{n+1},\label{a1} \\  \label{a2}
\bar \n_Y  E_i &=& -\bar g (Y, \bar E_{n+1}) E_i + \bar g (Y, \bar E_i) E_{n+1}, \quad \forall i=1,2, \cdots, n, \\ \label{a3}
\bar \n_Y E_{n+1}  &=& - \frac 1{x_{n+1}} Y,\label{dirEn}
\end{eqnarray}
for any vector field $Y$ in $\hh^{n+1}$.

The following propositions play a crucial role in this section.
\begin{prop}[{\rm \cite[Proposition 2.2]{GWX3}}]\label{lem2.1} \

$(\rm i)$  The vector fields  \,$x$ and $\{E_i\}_{i=1}^{n}$ are Killing vector fields in $\hh^{n+1}$, i.e,
\begin{eqnarray}\label{xKilling}
\frac12\big( \bar g (\bar \n_A x, E_B) + \bar g (\bar \n_B x, E_A)
 ) =\frac12\big( \bar g( \bar \n_A E_{i}, E_B)+
 \bar g( \bar \n_B E_{i}, E_A)
 \big)=0.
\end{eqnarray}

$(\rm {ii})$\,$E_{n+1}$ is a conformal Killing vector field in $\hh^{n+1}$, i.e,
\begin{eqnarray}\label{En-Killing}
\frac12\big (\bar g( \bar \n_A E_{n+1}, E_B)+
\bar g(\bar \n_B E_{n+1}, E_A)
\big)= -\frac{1}{x_{n+1}}\bar{g}_{AB}.
\end{eqnarray}
Here $\bar \n_A = \bar \n _{E_A}$ and $\bar g_{AB}= \bar g(E_A, E_B)$.
\end{prop}

Now we recall a conformal Killing vector field $X_{n+1}$ and a function $V_{n+1}$ in $\hh^{n+1}$ from \cite{GX2} that we will use
later. Denote
  \begin{equation}\label{cfkill2}
     X_{n+1}=x-E_{n+1}, \quad V_{n+1}=\frac{1}{x_{n+1}}.
  \end{equation}
From Proposition \ref{lem2.1}, we have the following important properties
\begin{prop}[{\rm \cite[Proposition 2.1]{GX2}}]\label{xaa}\

$(\rm i)$\,$X_{n+1}$ is a conformal Killing vector field with $\frac{1}{2}\mathcal{L}_{X_{n+1}}\bar{g}=V_{n+1}\bar{g}$, namely
\begin{eqnarray}\label{XXaeq1}
\frac12\big[\bar{g}(\bar{\nabla}_{E_A} X_{n+1},E_B)+\bar{g}(\bar{\nabla}_{E_B} X_{n+1},E_A)\big]=V_{n+1}\bar{g}_{AB}.
\end{eqnarray}

 $ ({\rm ii})$ $X _{n+1}\mid_{\mathcal{H}}$ is a tangential vector field on $\mathcal{H}$, i.e.,
\begin{eqnarray}\label{XXaeq2}
\bar{g}(X_{n+1}, \bar{N})=0 \quad\text{\rm on}\,\,\mathcal{H}.
\end{eqnarray}
\end{prop}
\begin{prop}[{\rm \cite[Proposition 2.2]{GX2}}]\label{xaa2}  $V_{n+1}$ satisfies the following properties:
\
\begin{eqnarray}\label{va2}
  \bar{\n}^2 V_{n+1}&=& V_{n+1}  \bar{g } \quad\quad\text{in}\,\, \hh^{n+1},\\
\partial_{\bar{N}}V_{n+1}&=& V_{n+1} \quad  \quad\,\,\text{on}\,\, \mathcal{H}.
\end{eqnarray}
\end{prop}

\begin{prop}\label{stable-necess}Assume $0\leq r\leq n-1$, then any totally umbilical hypersurface supported on a horosphere $\mathcal{H}$ is a stable $(r+1)$-th capillary hypersurface.
\end{prop}
\begin{proof}
Let $\tilde{\Sigma}$ be a totally umbilical hypersurface supported on $\mathcal{H}$ with a contact angle $\theta$. Suppose the principal curvatures of $\tilde{\Sigma}$ are all equal to a certain nonnegative constant $\tilde{\Lambda}$. It is easy to check that
\begin{equation*}
  \sigma_{i}=\binom{n}{i}\tilde{\Lambda}^{i}\quad \text{for any}\quad i=0,1,\cdots, n.
\end{equation*}
It follows from Lemma \ref{Pr-prop} that
\begin{eqnarray*}
L_{r}\varphi&=&\binom{n-1}{r}\tilde{\Lambda}^{r}\Delta \varphi,\\
tr(P_{r}h^{2})&=&n\binom{n-1}{r}\tilde{\Lambda}^{r+2},\\
tr(P_{r})&=&(n-r)\binom{n}{r}\tilde{\Lambda}^{r}.
\end{eqnarray*}
Hence for any $\varphi\in\mathcal{F}$ by \eqref{stab-ineq} we have
\begin{eqnarray*}\label{staqq}
\mathcal{E}_{r+1}''(0)&=&-(n-r)\binom{n}{r+1}^{-1}\int_{\tilde{\Sigma}}\vp[L_{r}\varphi+tr(P_{r}h^2)\varphi-tr(P_{r})\varphi] dA\\
&=&-\frac{(r+1)(n-r)}{n}\tilde{\Lambda}^{r}\int_{\tilde{\Sigma}}\vp(\Delta\varphi+n\tilde{\Lambda}^{2}\varphi-n\varphi) dA\nonumber\\
&=&\frac{(r+1)(n-r)}{n}\tilde{\Lambda}^{r}\mathcal{E}_{1}''(0).\nonumber
\end{eqnarray*}
From \cite[Proposition 2.5]{GWX3} that any totally umbilical capillary hypersurface supported on the horosphere $\mathcal{H}$ is stable. Therefore, we have $\tilde{\Sigma}$ is stable $(r+1)$-th capillary hypersurface.
\end{proof}

\

\subsection{Key formulae for $(r+1)$-th capillary hypersurfaces supported on a horosphere}\

In this subsection we show some useful facts about $(r+1)$-th capillary hypersurfaces supported on a horosphere that we will
use later. Let $P_{r}^{\mu\mu}=\sigma_{r}(h|h_{\mu\mu})$ be the $r$-th mean curvature deleting $h(\mu,\mu)$ component from the second fundamental form $h$. For simplicity of the notation, we will omit writing the volume form $dA$ on $M$ and the area form $ds$ on $\p M$.
\begin{prop}\label{prop-integral}Let $x: M\to\hh^{n+1}$ be an isometric immersion supported on $\mathcal{H}$. Assume $x(M)$ intersects $\mathcal{H}$ at a constant contact angle $\th \in (0, \pi)$. Then for any $0\leq r\leq n-1$,
\begin{eqnarray}
&&-(r+1)\int_{M}\bar{g}(x,\nu)\sigma_{r+1}dA =\int_{\partial M}P_{r}^{\mu\mu}(\cos\theta\bar{g}(x,\bar \nu)-\sin\theta)\,ds.\label{boundary2}
\end{eqnarray}
\end{prop}
\begin{proof}
Let $x^{T}:=x-\bar{g}(x,\nu)\nu=\sum_{i=1}^{n}(x^{T})^{i}e_{i}$, where $\{e_{i}\}_{i=1}^{n}$ is an orthonormal basis of $M$. Thus
\begin{equation*}
(x^{T})^{i}=\bar{g}(x-\bar{g}(x,\nu)\nu,e_{i})=\bar{g}(x,e_{i}).
\end{equation*}
Since $x$ is Killing vector field in $\mathbb{H}^{n+1}$, we see from Lemma \ref{Pr-prop} that
\begin{equation*}\label{balance1}
\div_{M}(P_{r}\circ x^{T})=P^{ij}_{r}\nabla_{j}(\bar{g}(x,e_{i}))=-\bar{g}(x,\nu)tr(P_{r}h)=-(r+1)\sigma_{r+1}\bar{g}(x,\nu).
\end{equation*}
Integration by parts gives
\begin{eqnarray*}
&&\int_{M}\div_{M}(P_{r}\circ x^{T})=\int_{\partial M}P_{r}(x^{T},\mu)=\int_{\partial M}P_{r}^{\mu\mu}\bar{g}(x,\mu)\\
&=&\int_{\partial M}P^{\mu\mu}_{r}\bar{g}(x,\cos\theta\bar \nu +\sin\theta \bar{N})=\int_{\partial M}P^{\mu\mu}_{r}(\cos\theta\bar{g}(x,\bar \nu)-\sin\theta), \nonumber
\end{eqnarray*}
where we have used \eqref{principal-direc}, \eqref{nu0} and $\bar{g}(x,\bar{N})=-1$ on $\partial M$.
 \end{proof}

Next we will derive another important integral identity.
\begin{prop}\label{Proposition-x} Let $x: M\to\hh^{n+1}$ be an isometric immersion supported on $\mathcal{H}$. Assume $x(M)$ intersects $\mathcal{H}$ at a constant contact angle $\th \in (0, \pi)$. Then for any $0\leq r\leq n$,
\begin{eqnarray} \label{boundary3}
	&&(n-r)\int_{M}\sigma_{r}\bar{g}(x,\nu)dA=\int_{\partial M}P_{r}^{\mu\mu}\bar{g}(x,\bar \nu )\,ds.
\end{eqnarray}
In particular, if $r=0$ we have
\begin{eqnarray} \label{boundary3-r0}
	&&n\int_{M}\bar{g}(x,\nu)dA=\int_{\partial M}\bar{g}(x,\bar \nu )\,ds.
\end{eqnarray}

\end{prop}

\begin{proof}
In order to prove \eqref{boundary3}, we consider a vector field $Z$ on $M$ as follows:
\begin{equation*}\label{zn}
  Z=\bar{g}(x,\nu)  \bar{E}_{n+1}-\bar{g}(\bar{E}_{n+1},\nu)x.
\end{equation*}
Recall that  $\bar{E}_{n+1}=x_{n+1}E_{n+1}$.
Along $\partial M$ we have
\begin{eqnarray*}\label{zn2}
  \bar{g}(Z,\mu)&=&\bar{g}(x,\nu)\bar{g}(\bar{E}_{n+1},\mu)-\bar{g}(\bar{E}_{n+1},\nu)\bar{g}(x,\mu)\\
  &=&-\bar{g}(x,\nu)\bar{g}(\bar{N},\mu)+\bar{g}(\bar{N},\nu)\bar{g}(x,\mu)\nonumber\\
  &=&-\sin\theta\bar{g}(x,\nu)-\cos\theta\bar{g}(x,\mu)\nonumber\\
  &=&-\bar{g}(x,\bar \nu),\nonumber
\end{eqnarray*}
where we have used \eqref{Nbar}, \eqref{nubar} and the fact $\bar{N}=-\bar{E}_{n+1}$ on $\partial M$. By integrating by parts we obtain
\begin{equation*}\label{diver}
-\int_{\partial M}P_{r}^{\mu\mu}\bar{g}(x,\bar \nu)=\int_{\partial M}P_{r}^{\mu\mu}\bar{g}(Z,\mu)=\int_{\partial M}P_{r}(Z^{T},\mu)=\int_{M}\div_{M}(P_{r}\circ Z^{T}).
\end{equation*}

Now we claim that
\begin{equation}\label{diver2}
\div_{M}(P_{r}\circ Z^{T})=-(n-r)\sigma_{r}\bar{g}(x,\nu).
\end{equation}
Thus Proposition \ref{Proposition-x} follows from claim \eqref{diver2}. Next we will show the claim \eqref{diver2}. First we observe that $Z$ is tangential, i.e, $\bar g (Z, \nu)=0$, which implies $ \div_{M}(P_{r}\circ Z^{T})=\div_{M}(P_{r}\circ Z)$. From \eqref{a1}, we see that $Z$ can be expressed as $Z= \bar \n_\nu x.$

Let $\{e_i\}_{i=1}^n$ be an othonormal basis of $M$. By constant sectional curvature of $\hh^{n+1}$ being $-1$, we have
	\begin{eqnarray*}
&&\bar{g}(\bar{\nabla}_{e_{i}}(\bar{\nabla}_{\nu}x), e_{j})\\
	&=&\bar g(\bar \n_{\nu} (\bar \n_{e_i} x), e_j)-\bar g(\bar \n_{[\nu, e_i]}x, e_j)- \bar g(\bar R(\nu, e_i)x, e_j)\nonumber
	\\&=&\bar \n_{\nu}(\bar g(\bar \n_{e_i} x, e_j))-\bar g(\bar \n_{e_i} x,  \bar \n_{\nu} e_j)-\bar g(\bar \n_{[\nu, e_i]} x, e_j)-\delta_{ij}\bar g(x, \nu)\nonumber
	\\&=&\bar \n_{\nu}(\bar g(\bar \n_{e_i} x, e_j))-\bar g(\bar \n_{e_i} x,  \bar \n_{e_j}\nu+[\nu, e_j])-\bar g(\bar \n_{[\nu, e_i]} x, e_j)-\delta_{ij}\bar g(x, \nu)\nonumber
	\\&=&\bar \n_{\nu}(\bar g(\bar \n_{e_i} x, e_j))-h_{jk}\bar g(\bar{\nabla}_{e_{i}}x, e_{k})-\bar g(\bar \n_{e_i} x,[\nu, e_j])-\bar g(\bar \n_{[\nu, e_i]} x, e_j)-\delta_{ij}\bar{g}(x,\nu).\nonumber
	\end{eqnarray*}
By utilizing $x$ is Killing vector field, we obtain
	\begin{eqnarray}\label{xmu3o9i}
P_{r}^{ij}\bar{g}(\bar{\nabla}_{e_{i}}Z, e_{j})&=&-P_{r}^{ij}h_{jk}\bar g(\bar{\nabla}_{e_{i}}x, e_{k})-P_{r}^{ij}\mathcal{L}_{x}\bar{g}([\nu,e_{i}],e_{j})-tr(P_{r})\bar{g}(x,\nu)\\
&=&-(n-r)\sigma_{r}\bar{g}(x,\nu),\nonumber
	\end{eqnarray}
where we used the fact that $P_{r}^{ij}h_{jk}=P_{r}^{kj}h_{ji}$ from \eqref{Pr}. Thus we have claim \eqref{diver2} and the proof is completed.
\end{proof}

As a consequence we get the following significant integral identity.
\begin{corollary}Let $x: M\to\hh^{n+1}$ be an immersed constant $(r+1)$-th mean curvature hypersurface supported on $\mathcal{H}$. Assume $x(M)$ intersects $\mathcal{H}$ at a constant contact angle $\th \in (0, \pi)$. Then for any $0\leq r\leq n-1$,
  	\begin{eqnarray}
	&&\int_{\partial M}P_{r}^{\mu\mu}(-\sin\theta+\cos\theta \bar{g}(x,\bar \nu )+\bar{g}(x,\bar \nu )h(\mu,\mu))\,ds=0.\label{integral-b}
	\end{eqnarray}
	\end{corollary}
	\begin{proof}By \eqref{boundary2} and \eqref{boundary3-r0}, we have
\begin{eqnarray}\label{xmu-87}
\int_{\partial M}P_{r}^{\mu\mu}(-\sin\theta+\cos\theta \bar{g}(x,\bar{\nu}))ds&=&-(r+1)\int_{M}\sigma_{r+1}\bar{g}(x,\nu)dA\\
&=&-\frac{r+1}{n}\sigma_{r+1}\int_{\partial M}\bar{g}(x,\bar{\nu})ds.\nonumber
\end{eqnarray}
Utilizing the fact that $\sigma_{r+1}=P_{r+1}^{\mu\mu}+P_{r}^{\mu\mu}h(\mu,\mu)$ from \eqref{Pr}, we get
\begin{eqnarray}\label{xmu-okuh}
\int_{\partial M}P_{r}^{\mu\mu}h(\mu,\mu)\bar{g}(x,\bar{\nu})ds&=&\int_{\partial M}(\sigma_{r+1}-P_{r+1}^{\mu\mu})\bar{g}(x,\bar{\nu})ds\\
&=&\sigma_{r+1}\int_{\partial M}\bar{g}(x,\bar{\nu})ds-(n-r-1)\int_{M}\sigma_{r+1}\bar{g}(x,\nu)dA,\nonumber
\end{eqnarray}
where we used \eqref{boundary3} in the last equality. Combining \eqref{xmu-87} and \eqref{xmu-okuh} we obtain
\begin{eqnarray*}
&&\int_{\partial M}P_{r}^{\mu\mu}(-\sin\theta+\bar{g}(x,\bar{\nu})\cos\theta+\bar{g}(x,\bar{\nu})h(\mu,\mu))ds\\
&=&\frac{n-r-1}{n}\sigma_{r+1}\left[\int_{\partial M}\bar{g}(x,\bar{\nu})ds-n\int_{M}\bar{g}(x,{\nu})dA\right]=0,\nonumber
\end{eqnarray*}
where the last equality we used \eqref{boundary3-r0} again.
\end{proof}

Now we use the conformal Killing vector field $X_{n+1}$ to establish a higher-order Minkowski-type formula, which is very crucial for the study of stable $(r+1)$-th capillary hypersurfaces supported on a horosphere.
\begin{prop}\label{Minkhorol}  Let $x: M\to  \hh^{n+1}$ be an isometric immersion supported on $\mathcal{H}$. Assume  $x(M)$ intersects $\mathcal{H}$ at a constant contact angle $\th \in (0, \pi)$. Then
\begin{eqnarray}\label{Mink1}
\quad\int_M [(V_{n+1}-\cos \theta \bar{g}(x,\nu))H_{r}-\bar{g}(X_{n+1},\nu)H_{r+1}]\,dA=0,\quad \text{for any}\,\,\,0\leq r\leq n-1.
\end{eqnarray}
\end{prop}
\begin{proof}
The proof could be seen in \cite[Proposition 4.]{CP}, we here give a proof for completeness. Let $X_{n+1}^{T}:=X_{n+1}-\bar{g}(X_{n+1},\nu)\nu=\sum_{i=1}^{n}(X_{n+1}^{T})^{i}e_{i}$, where $\{e_{i}\}_{i=1}^{n}$ is an orthonormal basis of $M$. Then
\begin{equation*}
(X_{n+1}^{T})^{i}=\bar{g}(X_{n+1}-\bar{g}(X_{n+1},\nu)\nu,e_{i})=\bar{g}(X_{n+1},e_{i}).
\end{equation*}
From Lemma \ref{Pr-prop} and \eqref{XXaeq1}, we find
\begin{eqnarray}\label{LHS-mink}
 \div_{M}(P_{r}\circ X_{n+1}^{T})&=&P^{ij}_{r}\nabla_{j}(\bar{g}(X_{n+1},e_{i}))\\
 &=&tr(P_{r})V_{n+1}-\bar{g}(X_{n+1},\nu)tr(P_{r}h)\nonumber\\
&=&(n-r)\sigma_{r}V_{n+1}-(r+1)\sigma_{r+1}\bar{g}(X_{n+1},\nu).\nonumber
\end{eqnarray}
Then by integration by parts, we get
\begin{eqnarray}\label{Mink1-iju}
  \int_{M}\div_{M}(P_{r}\circ X_{n+1}^{T})&=&\int_{\partial M}P_{r}(X^{T}_{n+1},\mu)=\int_{\partial M}P_{r}^{\mu\mu}\bar{g}(X_{n+1},\mu)\\
  &=&\cos\theta\int_{\partial M}P_{r}^{\mu\mu}\bar{g}(X_{n+1},\bar \nu )=\cos\theta\int_{\partial M}P_{r}^{\mu\mu}\bar{g}(x,\bar \nu ),\nonumber
\end{eqnarray}
where we used \eqref{principal-direc}, \eqref{mu0}, \eqref{XXaeq2} and $\bar{g}(E_{n+1},\bar \nu )=0$ along $\partial M$.
 Combining \eqref{LHS-mink}, \eqref{Mink1-iju} and \eqref{boundary3}, we complete the proof.

\end{proof}

Denote the $r$-th Jacobi operator to be $J_{r}:=L_{r}+tr(P_{r}h^{2})-tr(P_{r})$, where $L_{r}$ is given by \eqref{Lr-operator}. Recall the conformal Killing vector $X_{n+1}=x-E_{n+1}$. In order to investigate stability of $(r+1)$-th capillary hypersurface by the above higher-order Minkowski-type formula \eqref{Mink1}. We next calculate differential equations for $\bar g (x, \nu)$, $\bar g(E_{n+1}, \nu)$, $\bar g(X_{n+1}, \nu)$ and $V_{n+1}$ as follows.
\begin{prop}\label{prop4.66} Let $x: M\to\hh^{n+1}$ be a constant $(r+1)$-th mean curvature hypersurface.  Then for any $0\leq r\leq n-1$,
	\begin{eqnarray}
	J_{r} \bar g(x, \nu)&=&0,\label{xmu}\\
	J_{r}\bar g(E_{n+1}, \nu)&=&-(r+1)\sigma_{r+1}V_{n+1}-(n-r)\sigma_{r}\bar g(E_{n+1},\nu),\label{Ennu}\\
	J_{r}\bar g(X_{n+1}, \nu)&=&(r+1)\sigma_{r+1}V_{n+1}+(n-r)\sigma_{r}\bar g(E_{n+1},\nu),\label{Xmu1}\\
	J_{r}V_{n+1}&=&(r+1)\sigma_{r+1}\bar{g}(E_{n+1},\nu)+(\sigma_{1}\sigma_{r+1}-(r+2)\sigma_{r+2})V_{n+1}.\label{JV}
\end{eqnarray}
\end{prop}
\begin{proof}
	It is clear that  \eqref{Xmu1}  follows from \eqref{xmu} and \eqref{Ennu}. Therefore, we only need to show \eqref{xmu}, \eqref{Ennu} and \eqref{JV} one by one.

For a fixed point $p\in M$ let $\{e_i\}_{i=1}^{n}$ at $p$ be the local orthonormal basis and $\nabla_{e_{i}}e_{j}|_{p}=0$.
By \eqref{xKilling}, we  calculate at $p$,
	\begin{eqnarray*}
	e_{i}\bar{g}(x,\nu)=\bar{g}(x,\bar{\nabla}_{e_{i}}\nu)+\bar{g}(\bar{\nabla}_{e_{i}}x,\nu)=\bar{g}(x,\bar{\nabla}_{e_{i}}\nu)-\bar{g}(\bar{\nabla}_{\nu}x, e_{i}).
	\end{eqnarray*}
	It follows that
	\begin{eqnarray}\label{xmu33}
	&&L_{r}\bar{g}(x,\nu)=P_{r}^{ij}\bar{g}(x,\nu)_{,ij}\\ &=&P_{r}^{ij}\left[\bar{g}(\bar{\nabla}_{e_{j}}x,\bar{\nabla}_{e_{i}}\nu)+\bar{g}(x,\bar{\nabla}_{e_{j}}(\bar{\nabla}_{e_{i}}\nu))-\bar{g}(\bar{\nabla}_{e_{j}}(\bar{\nabla}_{\nu}x), e_{i})-\bar{g}(\bar{\nabla}_{\nu}x, \bar{\nabla}_{e_{j}}e_{i})\right]\nonumber\\
	&=&P_{r}^{ij}[h_{ik}\bar{g}(\bar{\nabla}_{e_{j}}x, e_{k})+\bar{g}(x,\nabla h_{ij}-h^{2}_{ij}\nu)-\bar{g}(\bar{\nabla}_{e_{j}}(\bar{\nabla}_{\nu}x), e_{i})+h_{ij}\bar{g}(\bar{\nabla}_{\nu}x, \nu)]\nonumber\\
	&=&\bar{g}(x,\nabla\sigma_{r+1})-tr(P_{r}h^{2})\bar{g}(x,\nu)-P_{r}^{ij}\bar{g}(\bar{\nabla}_{e_{j}}(\bar{\nabla}_{\nu}x), e_{i}),\nonumber\\
	&=&-tr(P_{r}h^{2})\bar{g}(x,\nu)-P_{r}^{ij}\bar{g}(\bar{\nabla}_{e_{j}}(\bar{\nabla}_{\nu}x), e_{i}),\nonumber	
\end{eqnarray}
	where we have used \eqref{xKilling}  and  the fact $\sigma_{r+1}$ is constant.
		In \eqref{xmu3o9i}, we have proved that
	\begin{eqnarray}\label{xmu34'}
	&&-P_{r}^{ij}\bar{g}(\bar{\nabla}_{e_{j}}(\bar{\nabla}_{\nu}x), e_{i})=(n-r)\sigma_{r}\bar g(x, \nu).
	\end{eqnarray}
 \eqref{xmu} follows from  \eqref{xmu33} and \eqref{xmu34'}.
	
	Using \eqref{a3} and \eqref{En-Killing}, we can directly check that
	\begin{eqnarray}\label{xmuLr}
	\quad\quad\,\, L_{r}\bar{g}(E_{n+1},\nu)&=&P_{r}^{ij}\bar{g}(E_{n+1},\nu)_{,ij}
	=P_{r}^{ij}[e_{j}(\bar{g}(\bar{\nabla}_{e_{i}}E_{n+1},\nu)+\bar{g}(E_{n+1},\bar{\nabla}_{e_{i}}\nu))]\\
	&=&P_{r}^{ij}(e_{j}\bar{g}(E_{n+1},\bar{\nabla}_{e_{i}}\nu))
	=P_{r}^{ij}(\bar{g}(\bar{\nabla}_{e_{j}}E_{n+1}, \bar{\nabla}_{e_{i}}\nu)+\bar{g}(E_{n+1}, \bar{\nabla}_{e_{j}}(\bar{\nabla}_{e_{i}}\nu)))\nonumber\\
	&=&P_{r}^{ij}h_{ik}\bar{g}(\bar{\nabla}_{e_{j}}E_{n+1}, e_{k})+P_{r}^{ij}\bar{g}(E_{n+1}, \nabla h_{ij}-h^{2}_{ij}\nu)\nonumber\\
	&=&-tr(P_{r}h)V_{n+1}+\bar{g}(E_{n+1},\nabla\sigma_{r+1})-tr(P_{r}h^{2})\bar{g}(E_{n+1}, \nu),\nonumber\\
	&=&-tr(P_{r}h)V_{n+1}-tr(P_{r}h^{2})\bar{g}(E_{n+1}, \nu),\nonumber
\end{eqnarray}
	which implies \eqref{Ennu}. From Proposition \ref{xaa2}, we see that
	\begin{eqnarray}\label{Lrvn+1}
	L_{r}V_{n+1}&=&P_{r}^{ij}(V_{n+1})_{,ij}=P_{r}^{ij}(\bar{\nabla}_{ij}V_{n+1}-h_{ij}\bar{\nabla}_{\nu}V_{n+1})\\
&=&tr(P_{r})V_{n+1}+tr(P_{r}h)\bar{g}(E_{n+1},\nu),\nonumber
\end{eqnarray}
which obtains \eqref{JV}. The proof is finished.
\end{proof}

Now we compute the boundary equations of the corresponding geometric quantities.
\begin{prop}[{\rm \cite[Proposition 3.6]{GWX3}}]\label{prop-boundary} Let $x: M\to  \hh^{n+1}$ be an isometric immersion supported on $\mathcal{H}$. Assume  $x(M)$ meets $\mathcal{H}$ at a constant contact angle $\th \in (0, \pi)$. Then along $\p M$, we have
\begin{eqnarray}\label{Hyp-dddd1}
{\nabla}_{\mu}(V_{n+1}-\cos\theta\bar{g}(E_{n+1},\nu))&=&q(V_{n+1}-\cos\theta\bar{g}(E_{n+1},\nu)),\label{boundary-111}\\
{\nabla}_{\mu}\bar{g}(X_{n+1},\nu)&=&q\bar{g}(X_{n+1},\nu),\label{boundary-222}\\
{\nabla}_{\mu}\bar{g}(x,\nu)&=&\bar{g}(x, \bar \nu )+h(\mu, \mu)\bar{g}(x, \mu),\label{Hyp-dddd22}
\end{eqnarray}
where $q$ is defined by \eqref{qsdsdvb}.
\end{prop}

\subsection{Rigidity for stable $(r+1)$-th capillary hypersurfaces supported on a horosphere}\

In this subsection we will show a uniqueness result for stable $(r+1)$-th capillary hypersurfaces supported on a horosphere $\mathcal H$.
For notation simplicity, we denote
\begin{equation}\label{u-define}
  u:=V_{n+1}-\cos \theta \bar{g}(x,\nu).
\end{equation}
\begin{prop}\label{prop-3.1u}Assume $0\leq r\leq n-1$. Let $x: M\to\hh^{n+1}$ be a constant $(r+1)$-th mean curvature hypersurface with boundary supported on $\mathcal{H}$. Assume $x(M)$ intersects $\mathcal{H}$ at a constant contact angle $\th \in (0, \pi)$.
Then $u$ satisfies
\begin{eqnarray}
J_{r}u&=&(r+1)\sigma_{r+1}\bar{g}(E_{n+1},\nu)+(\sigma_{1}\sigma_{r+1}-(r+2)\sigma_{r+2})V_{n+1},\,\,\,\,\text{in}\,\, M.\label{Hyp-bdy1uu}\\
{\nabla}_{\mu}u&=&qu,\,\,\,\,\text{on}\,\, \partial M.\label{Hyp-bdy1uuty}
\end{eqnarray}
\end{prop}
\begin{proof}
	\eqref{Hyp-bdy1uu} and \eqref{Hyp-bdy1uuty} follow from Proposition \ref{prop4.66} and Proposition \ref{prop-boundary} respectively.
\end{proof}

\begin{prop}\label{thmm4.frwsss112}
Assume $0\leq r\leq n-1$. Let $x: M\to\hh^{n+1}$ be a $(r+1)$-th capillary hypersurface with boundary supported on $\mathcal{H}$ at a constant contact angle $\theta\in(0,\pi)$. If $M$ is stable and there exists at least an elliptic point, then $\int_{M}u\,dA\neq0$.
\end{prop}
\begin{proof}
By contradiction, if $\int_{M}u\,dA=0$. Combining with \eqref{Hyp-bdy1uuty} we know that $u\in\mathcal{F}$. Now we choose $u$ as an admissible test function in \eqref{stable1-half}.
Therefore, by \eqref{u-define}, we have
\begin{eqnarray}\label{xeq1111uu}
0 &\le &-\int_M uJ_{r}u=-\int_M (V_{n+1}-\cos\theta \bar{g}(x,\nu))J_{r}u\\
&=&-\int_M V_{n+1}J_{r}u\, +\cos\theta\int_{M}\bar{g}(x,\nu)J_{r}u.\nonumber
\end{eqnarray}
We compute the last term of \eqref{xeq1111uu} by Green's  formula. From \eqref{xmu}, \eqref{principal-direc} and \eqref{Hyp-bdy1uuty}, we obtain
\begin{eqnarray}\label{xeq11120uu}
\int_{M}\bar{g}(x,\nu)J_{r}u&=&\int_{M}uJ_{r}\bar{g}(x,\nu)+\int_{\partial M}[\bar{g}(x,\nu)P_{r}(\nabla u,\mu)-uP_{r}(\nabla\bar{g}(x,\nu),\mu)]\\
&=&\int_{\partial M}P_{r}^{\mu\mu}u(q\cdot\bar{g}(x,\nu)-\nabla_{\mu}\bar{g}(x,\nu)).\nonumber
\end{eqnarray}
By \eqref{Hyp-dddd22}, \eqref{nu0} and \eqref{Nbar}, we see along $\partial M$
\begin{eqnarray}\label{xeq1113uu}
q\cdot\bar{g}(x,\nu)-\nabla_{\mu}\bar{g}(x,\nu)&=&\left(\csc\theta+\cot\theta h(\mu,\mu)\right)\bar{g}(x,\nu)-(\bar{g}(x,\bar \nu )+h(\mu, \mu)\bar{g}(x,\mu))\\
&=&-\cot\theta \bar{g}(x,\bar{N})-{\csc\theta} \bar{g}(x,\bar{N})h(\mu,\mu)\nonumber\\
&=&\cot\theta+{\csc\theta}  h(\mu,\mu),\nonumber
\end{eqnarray}
in the last equality we have used $\bar{g}(x,\bar{N})=-1$ on $\partial M$.
Instituting \eqref{xeq11120uu}-\eqref{xeq1113uu} into \eqref{xeq1111uu}, we get
\begin{equation}\label{stability-2uu}
  \int_M V_{n+1}J_{r}u\, -\cos\theta\int_{\partial M}P_{r}^{\mu\mu}u(\cot\theta+{\csc\theta}  h(\mu,\mu))\leq0.
\end{equation}

Now we introduce an auxiliary function
\begin{equation}\label{auxiuu}
  \Phi:=-\bar{g}(E_{n+1},\nu).
\end{equation}
By \eqref{xmuLr}, we obtain
\begin{equation}\label{auxi2uu}
  L_{r}\Phi=-L_{r}\bar{g}(E_{n+1},\nu)=V_{n+1}tr(P_{r}h)+\bar{g}(E_{n+1},\nu)tr(P_{r}h^{2})
\end{equation}
Note from \eqref{nu0} and \eqref{dirEn} that
\begin{eqnarray}\label{Phi-boun1uu}
\Phi|_{\p M}=-\cos\theta \quad\text{and} \quad\nabla_{\mu}\Phi=\sin\theta h(\mu,\mu).
\end{eqnarray}
Inserting \eqref{auxi2uu}-\eqref{Phi-boun1uu} into the following identity
\begin{eqnarray*}
&&\int_{M}[\Phi L_{r}\Phi+P_{r}(\nabla \Phi,\nabla \Phi)]=\int_{M}\frac{1}{2}L_{r}\Phi^{2}=\int_{\partial M}\Phi P_{r}(\nabla\Phi,\mu),
\end{eqnarray*}
we get an integral identity
\begin{equation}\label{xeq211uu}
\int_{M}(-V_{n+1}tr(P_{r}h)-\bar{g}(E_{n+1},\nu)tr(P_{r}h^{2}))\bar{g}(E_{n+1},\nu)+\int_{M}P_{r}(\nabla \Phi,\nabla \Phi)=-\cos\theta\sin\theta\int_{\partial M}P_{r}^{\mu\mu}h(\mu,\mu).
\end{equation}
Here we used $\mu$ is a principal direction by \eqref{principal-direc}.

Using \eqref{nu0}, on $\partial M$ we have
\begin{equation}\label{X-inteuuuu}
  u=V_{n+1}-\cos\theta\bar{g}(x,\nu)=V_{n+1}-\cos\theta(\cos\theta+\sin\theta\bar{g}(x,\bar{\nu}))=\sin\theta(\sin\theta-\cos\theta\bar{g}(x,\bar{\nu})).
\end{equation}

By adding \eqref{xeq211uu} to \eqref{stability-2uu} and applying \eqref{X-inteuuuu} and \eqref{Hyp-bdy1uu} we get
\begin{eqnarray}\label{xeq2112-iii}
0&\ge &\int_{M}[\bar{g}(E_{n+1}^{T},E_{n+1}^{T})tr(P_{r}h^{2})+P_{r}(\nabla \Phi,\nabla \Phi)]\\
&&-\cos^{2}\theta\int_{\partial M}P_{r}^{\mu\mu}\left(\sin\theta-\cos\theta\bar{g}(x,\bar{\nu})-\bar{g}(x,\bar{\nu})h(\mu,\mu)\right).\nonumber\\
&=&\int_{M}\left[\bar{g}(E_{n+1}^{T},E_{n+1}^{T})tr(P_{r}h^{2})+P_{r}(\nabla \Phi,\nabla \Phi)\right], \nonumber
\end{eqnarray}
where the last equality we used \eqref{integral-b}.

Since there exists an elliptic point on $M$, we know that $H_{r+1}>0$. From Proposition \ref{prop-elliptic}, one see $L_{i}$ is elliptic and $H_{i}>0$ for each $i=1,2,\cdots, r$.

Thus by Lemma \ref{Newton-Mac} we have
\begin{eqnarray}\label{inequality-big}
tr(P_{r}h^{2})&=&\sigma_{1}\sigma_{r+1}-(r+2)\sigma_{r+2}\\
&=&(r+1)\binom{n}{r+1}H_{1}H_{r+1}+(n-r-1)\binom{n}{r+1}(H_{1}H_{r+1}-H_{r+2})\nonumber\\
&>&0.\nonumber
\end{eqnarray}
From \eqref{xeq2112-iii}, \eqref{inequality-big} and \eqref{auxi2uu}, we obtain $\Phi$ is a constant on $M$, i.e.,
\begin{equation}\label{dsdsduu99a}
L_{r}\Phi=(r+1)\sigma_{r+1}V_{n+1}+(\sigma_{1}\sigma_{r+1}-(r+2)\sigma_{r+2})\bar{g}(E_{n+1},\nu)=0.
\end{equation}
Since $M$ is a compact hypersurface with boundary supported on $\mathcal{H}$, we have
\begin{equation*}
-\Phi=\bar{g}(E_{n+1},\nu)=\cos\theta>0\quad\text{on}\,\,M.
\end{equation*}
It follows from \eqref{dsdsduu99a} that
\begin{eqnarray*}
\sigma_{1}\sigma_{r+1}-(r+2)\sigma_{r+2}<0.
\end{eqnarray*}
We get a contradiction by \eqref{inequality-big}. Therefore, we conclude that $\int_{M}u\,dA\neq0$.
\end{proof}

In the following part we are ready to prove the classification for stable $(r+1)$-th capillary hypersurfaces supported on a horosphere $\mathcal{H}$. Inspired by the higher-order Minkowski-type formula \eqref{Mink1} and Proposition \ref{thmm4.frwsss112}, we have an admissible test function defined by
\begin{equation}\label{test-function}
\varphi_{n+1}:=\lambda u-\bar{g}(X_{n+1},\nu)H_{r+1},
\end{equation}
where $\lambda:=(\int_{M}u\,dA)^{-1}\int_{M}uH_{r}\,dA$ is constant and $u$ is given by \eqref{u-define}.

For the convenience, we denote
\begin{eqnarray}
\phi:&=&\lambda(r+1)\sigma_{r+1}-(n-r)\sigma_{r}H_{r+1},\label{htrr}\\
\psi:&=&\lambda (\sigma_{1}\sigma_{r+1}-(r+2)\sigma_{r+2})-(r+1)\sigma_{r+1}H_{r+1}.\label{hrr1}
\end{eqnarray}
In particular, if $r=0$ we have $\lambda=1$ and $\phi=0$ and $\psi=|h|^{2}-nH^{2}_{1}$.
\begin{prop}\label{prop-3.1}Let $x: M\to\hh^{n+1}$ be a constant $(r+1)$-th mean curvature hypersurface with boundary supported on $\mathcal{H}$. Assume $x(M)$ intersects $\mathcal{H}$ at a constant contact angle $\th \in (0, \pi)$.
Then $\varphi_{n+1}$ satisfies
\begin{eqnarray}
 J_{r}\varphi_{n+1}&=& \phi\bar{g}(E_{n+1},\nu)+\psi V_{n+1},\label{varphi1}\\
	{\nabla}_{\mu}\varphi_{n+1}&=&q\varphi_{n+1},\label{Hyp-bdy1}\\
\int_{M}\varphi_{n+1}\,dA&=&0.\label{bdy-zero1}
\end{eqnarray}
\end{prop}
\begin{proof}
	\eqref{varphi1} follows from \eqref{Hyp-bdy1uu} and \eqref{Xmu1}. \eqref{Hyp-bdy1} from \eqref{Hyp-bdy1uuty} and \eqref{boundary-222}. \eqref{bdy-zero1} is exactly from \eqref{Mink1} and the definition of $\lambda$.
\end{proof}

Now we are going to prove the rigidity result for stable $(r+1)$-th capillary hypersurfaces with boundary supported on a horosphere in $\mathbb{H}^{n+1}$ as follows.
\begin{theorem}\label{thmm4.1phi}
Assume $0\leq r\leq n-1$. Let $x: M\to\hh^{n+1}$ be a $(r+1)$-th capillary hypersurface supported on $\mathcal{H}$ at a constant contact angle $\theta\in(0,\pi)$. If $M$ is stable and there exists at least an elliptic point, then $M$ is totally umbilical.
\end{theorem}

\begin{proof} From Proposition \ref{thmm4.frwsss112}, we know that $\int_{M}u\,dA\neq0$. Thus by \eqref{Hyp-bdy1} and \eqref{bdy-zero1}, we can choose $\varphi_{n+1}$ as an admissible test function in \eqref{stable1-half}.
Therefore,
\begin{eqnarray}\label{xeq1111}
\qquad\,\,0 &\le &-\int_M\varphi_{n+1}J \varphi_{n+1}=-\int_M (\lambda V_{n+1}-\lambda\cos\theta \bar{g}(x,\nu)-\bar{g}(x-E_{n+1},\nu)H_{r+1})J_{r}\varphi_{n+1}\,  \\
&=&-\int_M (\lambda V_{n+1}+\bar{g}(E_{n+1},\nu)H_{r+1})J_{r}\varphi_{n+1}\, +(H_{r+1}+\lambda\cos\theta)\int_{M}\bar{g}(x,\nu)J_{r}\varphi_{n+1}.\nonumber
\end{eqnarray}
We compute the second term in the right hand side of \eqref{xeq1111} by Green's  formula. By \eqref{xmu} and \eqref{Hyp-bdy1}, we have
\begin{eqnarray}\label{xeq11120}
&&\int_{M}\bar{g}(x,\nu)J_{r}\varphi_{n+1}\\
&=&\int_{M}J_{r}\bar{g}(x,\nu)\varphi_{n+1}+\int_{\partial M}[\bar{g}(x,\nu)P_{r}(\nabla\varphi_{n+1},\mu)-\varphi_{n+1}P_{r}(\nabla\bar{g}(x,\nu),\mu)]\nonumber\\
&=&\int_{\partial M}P_{r}^{\mu\mu}(\bar{g}(x,\nu)\nabla_{\mu}\varphi_{n+1}-\varphi_{n+1}\nabla_{\mu}\bar{g}(x,\nu))\nonumber\\
&=&\int_{\partial M}P_{r}^{\mu\mu}\varphi_{n+1}(q\cdot\bar{g}(x,\nu)-\nabla_{\mu}\bar{g}(x,\nu)).\nonumber
\end{eqnarray}
From \eqref{Hyp-dddd22}, \eqref{nu0} and \eqref{Nbar}, we find along $\partial M$
\begin{eqnarray}\label{xeq1113}
q\cdot\bar{g}(x,\nu)-\nabla_{\mu}\bar{g}(x,\nu)&=&\left(\csc\theta+\cot\theta h(\mu,\mu)\right)\bar{g}(x,\nu)-(\bar{g}(x,\bar \nu )+h(\mu, \mu)\bar{g}(x,\mu))\\
&=&-\cot\theta \bar{g}(x,\bar{N})-{\csc\theta} \bar{g}(x,\bar{N})h(\mu,\mu)\nonumber\\
&=&\cot\theta+{\csc\theta}  h(\mu,\mu),\nonumber
\end{eqnarray}
where we used $\bar{g}(x,\bar{N})=-1$ on $\partial M$.

Instituting \eqref{xeq11120} and \eqref{xeq1113} into \eqref{xeq1111}, we get
\begin{equation}\label{stability-2}
  \int_M (\lambda V_{n+1}+\bar{g}(E_{n+1},\nu)H_{r+1})J_{r}\varphi_{n+1}\, -(H_{r+1}+\lambda\cos\theta)\int_{\partial M}P_{r}^{\mu\mu}\varphi_{n+1}(\cot\theta+{\csc\theta}  h(\mu,\mu))\leq0.
\end{equation}

Next we introduce a powerful auxiliary function to eliminate the integral boundary term of \eqref{stability-2}. Let
\begin{equation}\label{auxi}
  \Psi:=-H_{r+1}V_{n+1}-\lambda\bar{g}(E_{n+1},\nu).
\end{equation}
By \eqref{Lrvn+1} and \eqref{xmuLr}, we obtain
\begin{equation}\label{auxi2}
 L_{r}\Psi=\phi V_{n+1}+\psi\bar{g}(E_{n+1},\nu).
\end{equation}
From \eqref{nu0} we have
\begin{eqnarray}\label{Phi-boun1}
\Psi|_{\p M}=-H_{r+1}-\lambda\cos\theta.
\end{eqnarray}
By using \eqref{dirEn}, we can directly calculate
\begin{eqnarray}\label{Phi-boun2}
&&  \nabla_{\mu}\Psi=-\sin\theta(H_{r+1}-\lambda h(\mu,\mu)),
\end{eqnarray}
where we have used $\bar{g}(\bar{\nabla}_{\mu}E_{n+1},\nu)=0$ and $\bar{g}(E_{n+1},\mu)=-\sin\theta$ on $\partial M$.

Inserting \eqref{auxi}-\eqref{Phi-boun2} into the following integral identity
\begin{eqnarray*}
&&\int_{M}[\Psi L_{r}\Psi+P_{r}(\nabla\Psi,\nabla\Psi)]=\int_{M}\frac{1}{2}L_{r}\Psi^{2}=\int_{\partial M}\Psi P_{r}(\nabla\Psi,\mu),
\end{eqnarray*}
we get that
\begin{eqnarray}\label{xeq211}
&&\int_{M}(-H_{r+1}V_{n+1}-\lambda\bar{g}(E_{n+1},\nu))(\phi V_{n+1}+\psi\bar{g}(E_{n+1},\nu))+\int_{M}P_{r}(\nabla\Psi,\nabla\Psi)\\
&=&(H_{r+1}+\lambda\cos\theta)\sin\theta\int_{\partial M}P_{r}^{\mu\mu}(H_{r+1}-\lambda h(\mu,\mu)).\nonumber
\end{eqnarray}
Putting \eqref{xeq211} into \eqref{stability-2} and applying \eqref{varphi1} we have
\begin{eqnarray}\label{xeq2112}
\quad\quad\,\,0&\ge &\int_{M}\bar{g}(E_{n+1}^{T},E_{n+1}^{T})(\psi\lambda-\phi H_{r+1})+\int_{M}P_{r}(\nabla\Psi,\nabla\Psi)\\
&&-(H_{r+1}+\lambda\cos\theta)\int_{\partial M}P_{r}^{\mu\mu}\left[\sin\theta(H_{r+1}-\lambda h(\mu,\mu))+(\cot\theta+{\csc\theta}h(\mu,\mu))\varphi_{n+1}\right].\nonumber
\end{eqnarray}

Next we will show that the boundary term in the right hand side of \eqref{xeq2112} is zero. Indeed, by \eqref{nu0} and \eqref{XXaeq2} along $\partial M$
\begin{eqnarray*}
\varphi_{n+1}&=&\lambda V_{n+1}-\lambda\cos\theta \bar{g}(x,\nu)-\bar{g}(X_{n+1},\nu)H_{r+1}\\
&=&\lambda-\lambda\cos\theta(\cos\theta+\sin\theta\bar{g}(x,\bar{\nu}))-\sin\theta\bar{g}(x-E_{n+1},\bar{\nu})H_{r+1}\nonumber\\
&=&\sin\theta(\lambda\sin\theta-\lambda\cos\theta\bar{g}(x,\bar{\nu})-\bar{g}(x,\bar{\nu})H_{r+1}).\nonumber
\end{eqnarray*}
It follows that
\begin{eqnarray}\label{xeq2133uuueddw}
&&\sin\theta(H_{r+1}-\lambda h(\mu,\mu))+(\cot\theta+{\csc\theta}h(\mu,\mu))\varphi_{n+1}\\
&&=(H_{r+1}+\lambda\cos\theta)(\sin\theta-\cos\theta\bar{g}(x,\bar{\nu})-\bar{g}(x,\bar{\nu})h(\mu,\mu)).\nonumber
\end{eqnarray}
From \eqref{xeq2133uuueddw} and \eqref{integral-b}, we see that
\begin{equation}\label{kkij}
\int_{\partial M}P_{r}^{\mu\mu}\left[\sin\theta(H_{r+1}-\lambda h(\mu,\mu))+(\cot\theta+{\csc\theta}h(\mu,\mu))\varphi_{n+1}\right]=0.
\end{equation}
Putting \eqref{kkij} into \eqref{xeq2112}, we get
\begin{equation}\label{xeq2112uuuuy}
  0\ge\int_{M}[\bar{g}(E_{n+1}^{T},E_{n+1}^{T})(\psi\lambda-\phi H_{r+1})+P_{r}(\nabla \Psi,\nabla\Psi)].
\end{equation}

Since there exists an elliptic point on $M$, thus $H_{r+1}=constant>0$. From Proposition \ref{prop-elliptic}, the operator $L_{i}$ is elliptic and $H_{i}>0$ for each $i=1,2,\cdots, r$.

By the Newton-Maclaurin inequality \eqref{Newtonaa} we have
\begin{eqnarray}\label{xeq2112uinter}
\quad\psi\lambda-\phi H_{r+1}&=&\lambda^{2}(\sigma_{1}\sigma_{r+1}-(r+2)\sigma_{r+2})-2\lambda(r+1)\sigma_{r+1}H_{r+1}+(n-r)\sigma_{r}H_{r+1}^{2}\\
&=&\binom{n}{r+1}[\lambda^{2}(n-r-1)(H_{1}H_{r+1}-H_{r+2})+(r+1)\lambda^{2}H_{1}H_{r+1}]\nonumber\\
&&+\binom{n}{r+1}(r+1)[-2\lambda H_{r+1}^{2}+H_{r}H_{r+1}^{2}]\nonumber\\
&=&\lambda^{2}(n-r-1)\binom{n}{r+1}(H_{1}H_{r+1}-H_{r+2})\nonumber\\
&&+(r+1)\binom{n}{r+1}H_{r+1}^{2}\left(\frac{H_{1}}{H_{r+1}}\left(\lambda-\frac{H_{r+1}}{H_{1}}\right)^{2}+H_{r}-\frac{H_{r+1}}{H_{1}}\right)\nonumber\\
&\geq&0.\nonumber
\end{eqnarray}
Combining \eqref{xeq2112uinter} and \eqref{xeq2112uuuuy}, we obtain
\begin{eqnarray}\label{conc-1}
\bar{g}(E_{n+1}^{T}, E_{n+1}^{T})(\psi\lambda-\phi H_{r+1})=0,\quad\text{on}\,\, M,
\end{eqnarray}
and $\Psi$ is a constant on $M$, that is,
\begin{eqnarray}\label{conc-sss}
\Psi=-H_{r+1}V_{n+1}-\lambda\bar{g}(E_{n+1},\nu)=-H_{r+1}-\lambda\cos\theta,  \quad\text{on}\,\, M.
\end{eqnarray}

We claim that
\begin{equation}\label{sxxa1112}
  \psi\lambda-\phi H_{r+1}=0, \quad\text{on}\,\, M.
\end{equation}
In fact, the open set $U:=\{p\in M|\psi\lambda-\phi H_{r+1}\neq0\}$ is empty. If not, we have $\bar{g}(E_{n+1}^{T}, E_{n+1}^{T})=0$ on $U$ from \eqref{conc-1}. By the fact that $\bar{g}(\bar{E}_{n+1},\bar{E}_{n+1})=1$ we see that
\begin{equation}\label{jaaaaxqqqw}
\bar{g}(E_{n+1}, \nu)=\pm V_{n+1}, \quad\text{on}\,\, U.
\end{equation}
Since $\theta\in(0,\pi)$ and $H_{r+1}>0$, we combine \eqref{jaaaaxqqqw} and \eqref{conc-sss} to get $V_{n+1}$ is positive constant $c_{1}$ on $U$,
which means $U$ is lying on the horosphere $\{V_{n+1}=c_{1}\}$. On the other hand, from \eqref{xeq2112uinter} we know that $U^{c}$ is a part of the totally umbilical hypersurface. By the smoothness of $M$, we imply that $V_{n+1}$ is constant on the whole $M$. Thus $M$ lies on a horosphere in $\mathbb{H}^{n+1}$. Using \eqref{xeq2112uinter} again we have
\begin{equation*}
\psi\lambda-\phi H_{r+1}=0, \quad\text{on}\,\, M.
\end{equation*}
We get a contradiction, the claim \eqref{sxxa1112} is true. From \eqref{sxxa1112}, \eqref{xeq2112uinter} and Lemma \ref{Newton-Mac} we obtain that $M$ is a totally umbilical hypersurface.

\end{proof}

\appendix
\section{Proof of the first variational formula}\label{app}
The appendix is devoted to compute the first variational formula of $(r+1)$-th energy functional $\mathcal{E}_{r+1}$ and prove Theorem \ref{thm105-half}. Let $\mathbb{M}^{n+1}(K)$ be a complete simply-connected $(n+1)$-dimensional Riemannian manifold with constant sectional curvature $K$. We first study the evolution equations for several useful geometric quantities under the following flow in $\mathbb{M}^{n+1}(K)$:
\begin{equation}\label{flow100}
 \partial_{t} x=f \nu+T,
\end{equation}
where $T \in T M_{t}$.

\begin{prop}\label{formula111}
Along the general flow \eqref{flow100}, it holds that
\begin{itemize}
  \item [(1)]$\partial_{t} g_{i j}=2 f h_{i j}+\nabla_{i} T_{j}+\nabla_{j} T_{i}$
  \item [(2)]$\partial_{t} d A_{t}=(f H+\div T) d A_{t}$
  \item [(3)]$\partial_{t} \nu=-\nabla f+h\left(e_{i}, T\right) e_{i}$
  \item [(4)]$\partial_{t} h_{i j}=-\nabla_{i j}^{2} f+f (h_{i k} h_{j}^{k}-K\bar{g}_{ij})+\nabla_{T} h_{i j}+h_{j}^{k} \nabla_{i} T_{k}+h_{i}^{k} \nabla_{j} T_{k}$
  \item [(5)]$\partial_{t} h_{j}^{i}=-\nabla^{i} \nabla_{j} f-f(h_{j}^{k} h_{k}^{i}+K\bar{g}^{i}_{j})+\nabla_{T} h_{j}^{i}$
  \item [(6)]$\partial_{t} H=-\Delta f-(nK+|h|^{2})f+\nabla_{T} H$
  \item [(7)]$\partial_{t} F=-F_{i}^{j} \nabla^{i} \nabla_{j} f-f(F_{i}^{j} h_{j}^{k} h_{k}^{i}+KF^{j}_{i}\bar{g}^{i}_{j})+\nabla_{T} F$, for $F=F(h_{i}^{j})$, where $F_{j}^{i}:=\frac{\partial F}{\partial h_{i}^{j}}$
  \item [(8)]$\partial_{t} \sigma_{r}=-\frac{\partial \sigma_{r}}{\partial h_{i}^{j}} \nabla^{i} \nabla_{j} f-\left(\sigma_{1} \sigma_{r}-(r+1) \sigma_{r+1}\right)f-K(n-r+1)\sigma_{r-1}f+\nabla_{T} \sigma_{r}$.
\end{itemize}
\end{prop}
\begin{proof}
Let $\left\{e_{i}\right\}_{i=1}^{n}$ be an orthonormal basis of $T_{p}M$ for some point $p$. Denote $e_{i}(t):=(x(t,\cdot))_{*}(e_{i})$ and $Y(t):=\partial_{t}x(t,\cdot)$, then
we have $\bar{g}(e_{i}(t),\nu(t))=0$ and $[e_{i}(t),Y(t)]=0$. Recall the Gauss-Weingarten formula as follows:
\begin{equation*}\label{Gauss}
 \bar{\nabla}_{e_{i}} e_{j}=\nabla_{e_{i}} e_{j}-h_{i j} \nu, \quad \bar{\nabla}_{e_{i}} \nu=h_{i k} e_{k}.
\end{equation*}
We calculate that
\begin{eqnarray*}
\partial_{t} \bar{g}_{i j} &=&\partial_{t}\bar{g}(e_{i}(t), e_{j}(t)) \\
&=&\bar{g}(\bar{\nabla}_{Y}e_{i},e_{j})+\bar{g}(e_{i},\bar{\nabla}_{Y}e_{j})\nonumber\\
&=&\bar{g}(\bar{\nabla}_{e_{i}}Y,e_{j})+\bar{g}(e_{i},\bar{\nabla}_{e_{j}}Y )\nonumber\\
&=&\bar{g}(\bar{\nabla}_{e_{i}}(f\nu+T),e_{j})+\bar{g}(e_{i},\bar{\nabla}_{e_{j}}(f\nu+T))\nonumber\\
&=&2 f h_{i j}+\nabla_{i}T_{j}+\nabla_{j}T_{i}.\nonumber
\end{eqnarray*}
It follows that
\begin{equation*}
  \partial_{t} d A_{t} =\frac{1}{2} \bar{g}^{i j} \partial_{t} \bar{g}_{i j} d A_{t}
=\frac{1}{2}\bar{g}^{i j}\left(2 f h_{i j}+\nabla_{i}T_{j}+\nabla_{j}T_{i}\right) d A_{t}=(f H+\div T) d A_{t}.
\end{equation*}
Since $\bar{g}(\partial_{t} \nu, \nu)=0$, thus
\begin{eqnarray*}
\partial_{t} \nu &=&\bar{g}(\partial_{t}\nu,e_{i}(t))e_{i}(t)=-\bar{g}(\nu,\bar{\nabla}_{Y}e_{i})e_{i}=-\bar{g}(\nu,\bar{\nabla}_{e_{i}}Y)e_{i}=-\bar{g}(\nu,\bar{\nabla}_{e_{i}}(f\nu+T))e_{i}\\
&=&-(\bar{\nabla}_{e_{i}} f) e_{i}+h(e_{i},T)e_{i}=-\nabla f+h(e_{i},T)e_{i}.\nonumber
\end{eqnarray*}
We next compute
\begin{eqnarray}
\partial_{t} h_{i j}&=&\partial_{t}\bar{g}(\bar{\nabla}_{e_{i}(t)} \nu(t), e_{j}(t))\label{huning}\\
&=&\bar{g}(\bar{\nabla}_{Y}\bar{\nabla}_{e_{i}} \nu,e_{j})+\bar{g}(\bar{\nabla}_{e_{i}} \nu,\bar{\nabla}_{Y}e_{j})\nonumber\\
&=&\bar{g}(\bar{\nabla}_{e_{i}}\bar{\nabla}_{Y} \nu,e_{j})+\bar{g}(\bar{\nabla}_{[Y,e_{i}]}\nu,e_{j})+\bar{g}(\bar{R}(Y,e_{i})\nu,e_{j})+\bar{g}(\bar{\nabla}_{e_{i}} \nu,\bar{\nabla}_{e_{j}}Y)\nonumber\\
&=&\bar{g}(\bar{\nabla}_{e_{i}}(\partial_{t}\nu),e_{j})+\bar{R}(e_{j},\nu,Y,e_{i})+\bar{g}(\bar{\nabla}_{e_{i}} \nu,\bar{\nabla}_{e_{j}}(f\nu+T))\nonumber\\
&=&\bar{g}(\bar{\nabla}_{e_{i}}(-\nabla f+h(e_{k},T)e_{k}),e_{j})-K\bar{g}(Y,\nu)\bar{g}_{ij}+fh_{i}^{k}h_{kj}+h^{k}_{i}\nabla_{j}T_{k}\nonumber\\
&=&-\nabla^{2}_{ij}f-Kf\bar{g}_{ij}+fh_{i}^{k}h_{kj}+\bar{g}(\bar{\nabla}_{e_{i}}(h(e_{k},T)e_{k}),e_{j})+h^{k}_{i}\nabla_{j}T_{k}.\nonumber
\end{eqnarray}
Using Codazzi equation in a space form we have
\begin{eqnarray}
&&\bar{g}(\bar{\nabla}_{e_{i}}(h(e_{k},T)e_{k}),e_{j})\label{huning2}\\
&=&[\nabla_{e_{i}}h(e_{k},T)+h(\nabla_{e_{i}}e_{k},T)+h(e_{k},\nabla_{e_{i}}T)]\bar{g}_{kj}+h(e_{k},T)\bar{g}(\nabla_{e_{i}}e_{k},e_{j})\nonumber\\
&=&\nabla_{e_{i}}h(e_{j},T)+h(\nabla_{e_{i}}e_{j},T)+h^{k}_{j}\nabla_{i}T_{k}-h(e_{k},T)\bar{g}(e_{k},\nabla_{e_{i}}e_{j})\nonumber\\
&=&\nabla_{T}h(e_{j},e_{i})+h^{k}_{j}\nabla_{i}T_{k}.\nonumber
\end{eqnarray}
Combining \eqref{huning} and \eqref{huning2}, we get (4). It follows that
\begin{eqnarray*}
\partial_{t} h_{j}^{i}&=&\partial_{t}\left(\bar{g}^{i k} h_{k j}\right)\\
&=&-2 f h^{i}_{k} h^{k}_{j} -\nabla^{i}\nabla_{j} f+f h^{ik}h_{k j}-Kf\bar{g}^{i}_{j}+\nabla_{T}h^{i}_{j}\nonumber\\
&=&-\nabla^{i}\nabla_{j}f-f h^{ik}h_{k j}-Kf\bar{g}^{i}_{j}+\nabla_{T}h^{i}_{j}.\nonumber
\end{eqnarray*}
The last three assertions (6)-(8) follow directly from (5), we only need the following fact
\begin{equation*}
  \frac{\partial \sigma_{r}}{\partial h_{j}^{i}} h_{j}^{k} h_{k}^{i}=\sigma_{1} \sigma_{r}-(r+1) \sigma_{r+1}\quad\text{and}\quad\frac{\partial \sigma_{r}}{\partial h_{j}^{i}} \bar{g}_{j}^{i}=(n-r+1)\sigma_{r-1}.
\end{equation*}
\end{proof}
Next we let $\partial B$ be a totally umbilical hypersurface in $\mathbb{M}^{n+1}(K)$ with constant principal curvature $\kappa\in\mathbb{R}$. Let $x:M\rightarrow \mathbb{M}^{n+1}(K)$ be an immersed hypersurface with boundary $\partial M$ supported on $\partial B$. Assume the contact angle is constant $\theta\in(0,\pi)$ along $\partial M$.

From Proposition \ref{volume-preserving-1}, we choose $f\in\mathcal{F}$, then there exists an admissible volume-preserving and contact angle-preserving variation of $x$ with variational vector field $Y$ having $f \nu$ as its normal part. Namely,
\begin{eqnarray}
 Y &=&T+f\nu,\quad \text{on}\,\, M,\label{flow000}\\
\bar{g}(\nu,\bar{N}\circ x) &=&-\cos\theta,\quad\text{on}\,\, \partial M,\label{flow987}
\end{eqnarray}
where $T$ is tangent part of variational vector field $Y$.

Applying \eqref{Nbar}, \eqref{flow000} and admissible condition, we get
\begin{equation*}
  0=\bar{g}(Y,\bar{N})=\bar{g}(Y,\sin\theta\mu-\cos\theta\nu)=\sin\theta\bar{g}(Y,\mu)-f\cos\theta,\,\,\,\, \text{on}\,\, \partial M.
\end{equation*}
Therefore,
\begin{equation}\label{Ymu}
 \bar{g}(T,\mu)=\bar{g}(Y,\mu)=\cot\theta f,\,\,\,\, \text{on}\,\, \partial M.
\end{equation}
So we can define
\begin{equation}\label{jiut}
 Y=T+f\nu:=Y^{\partial M}+\cot\theta f\mu+f\nu,
\end{equation}
where $Y^{\partial M}$ denotes the tangent part of $Y$ to $\partial M$.

On the other hand, from \eqref{nubar} we see that $Y$ can be also expressed as follows:
\begin{equation}\label{Ymu4532}
  Y=Y^{\partial M}+\frac{f}{\sin\theta}(\cos\theta\mu+\sin\theta\nu)=Y^{\partial M}+\frac{f}{\sin\theta}\bar{\nu}.
\end{equation}

Recall that the $r$-th wetting area functional ${W}_{r}: (-\ep, \ep)\to\rr$ is inductively given by
\begin{eqnarray*}
{W}_{0}(t):&=&\int_{\partial M\times[0,t]}x^{*}dA_{\partial B},\\
{W}_{1}(t):&=&\frac{1}{n}\int_{\partial M}ds_{t},
\end{eqnarray*}
and for $2\leq r \leq n-1$,
\begin{equation}\label{wetting-def}
 {W}_{r}(t)=\frac{1}{n}\int_{\partial M}H^{\partial M}_{r-1}\,ds_{t}+\frac{r-1}{n-r+2}(K+\kappa^{2})W_{r-2}(t),
 \end{equation}
where $H^{\partial M}_{r-1}$ is $(r-1)$-th normalized mean curvature of closed hypersurface $\partial M$ in $\partial B$.

\begin{lemma}
For any $0\leq r \leq n-1$, we have the first variational formula of $ {W}_{r}(t)$
\begin{equation}\label{wetting-r}
 \frac{d}{dt} W_{r}(t)=\binom{n}{r}^{-1}\int_{\partial M}\sigma^{\partial M}_{r}\bar{g}(Y,\bar{\nu})\,ds_{t}=\frac{1}{\sin\theta}\binom{n}{r}^{-1}\int_{\partial M}\sigma^{\partial M}_{r}f\,ds_{t},
\end{equation}
where $\sigma^{\partial M}_{r}(\tilde{h})$ is $r$-th mean curvature of $\partial M$ in $\partial B$.
\end{lemma}
\begin{proof}
It is obviously true for $r=0$ and $r=1$. We next only consider $r\geq2$ case. Let $\partial M$ be an immersed closed hypersurface in $\partial B$ with a normal speed $\tilde{f}\bar{\nu}$. From \eqref{Ymu4532} we see $\tilde{f}=\frac{1}{\sin\theta}f$.
By Gauss equation we know that $\partial B$ has intrinsic constant sectional curvature $\tau=K+\kappa^{2}$.
From Proposition \ref{formula111}, we get
\begin{eqnarray}\label{wetting-r-2}
\quad&&\frac{d}{dt}\int_{\partial M}\sigma^{\partial M}_{r-1}\,ds_{t}\\
&=&\int_{\partial M}\left[-\frac{\partial\sigma^{\partial M}_{r-1}(\tilde{h})}{\partial \tilde{h}^{\beta}_{\alpha}}\tilde{\nabla}^{2}_{\alpha\beta}\tilde{f}-(\sigma_{1}^{\partial M}\sigma_{r-1}^{\partial M}-r\sigma_{r}^{\partial M})\tilde{f}-\tau(n-r+1)\sigma_{r-2}^{\partial M}\tilde{f}+\sigma_{r-1}^{\partial M}\sigma_{1}^{\partial M}\tilde{f}\right]ds_{t}\nonumber\\
&=&r\int_{\partial M}\sigma_{r}^{\partial M}\tilde{f}ds_{t}-\tau(n-r+1)\int_{\partial M}\sigma_{r-2}^{\partial M}\tilde{f}ds_{t}.\nonumber
\end{eqnarray}
By induction, we assume it is true for $r-2$ in \eqref{wetting-r}. Applying \eqref{wetting-def} and \eqref{wetting-r-2} we have
\begin{eqnarray}\label{wetting-r-3}
 \frac{d}{dt}W_{r}(t)&=&\frac{1}{n}\binom{n-1}{r-1}^{-1} \frac{d}{dt}\left(\int_{\partial M}\sigma_{r-1}^{\partial M}\,ds_{t}\right)+\frac{\tau(r-1)}{n-r+2} \frac{d}{dt}W_{r-2}(t)\\
 &=&\frac{1}{n}\binom{n-1}{r-1}^{-1}\left(r\int_{\partial M}\sigma_{r}^{\partial M}\tilde{f}ds_{t}-\tau(n-r+1)\int_{\partial M}\sigma_{r-2}^{\partial M}\tilde{f}ds_{t}\right)\nonumber\\
 &&+\frac{\tau(r-1)}{n-r+2}\binom{n}{r-2}^{-1}\int_{\partial M}\sigma_{r-2}^{\partial M}\tilde{f}ds_{t}\nonumber\\
 &=&\frac{r}{n}\binom{n-1}{r-1}^{-1}\int_{\partial M}\sigma_{r}^{\partial M}\tilde{f}ds_{t}\nonumber\\
 &=&\frac{1}{\sin\theta}\binom{n}{r}^{-1}\int_{\partial M}\sigma_{r}^{\partial M}fds_{t}.\nonumber
\end{eqnarray}
The proof is completed.

\end{proof}

\noindent{\bf Proof of Theorem \ref{thm105-half}.}\ For $r=0$ case, we can find in \cite{RS}.  In the following proof we only consider $1\leq r\leq n-1$ case. By Proposition \ref{formula111}, using integration by parts and the fact that $\mu$ is a principal direction from Proposition \ref{Principle}, we obtain
\begin{eqnarray}\label{var-formula}
&&\frac{d}{d t} \int_{M} \sigma_{r} d A_{t}\\
&=& \int_{M}\left[-\frac{\partial \sigma_{r}}{\partial h_{i}^{j}} \nabla^{i} \nabla_{j} f-f\left(\sigma_{1} \sigma_{r}-(r+1) \sigma_{r+1}\right)-K(n-r+1)f\sigma_{r-1}+\nabla_{{T}} \sigma_{r}\right] d A_{t} \nonumber\\
&&+\int_{M} \sigma_{r}\left(f \sigma_{1}+\operatorname{div}_{M}{T}\right) d A_{t}\nonumber\\
&=&(r+1) \int_{M} \sigma_{r+1}f d A_{t}-K(n-r+1)\int_{M}\sigma_{r-1}fdA_{t}+\int_{\partial M}\left(\sigma_{r}\bar{g}({T}, \mu)-\sigma_{r}^{\mu \mu} \nabla_{\mu} f\right)ds_{t}.\nonumber
\end{eqnarray}
Utilizing \eqref{Ymu}, \eqref{stable1-pre} and the principal curvature of $\partial B$ is $\kappa$, we get along $\partial M$,
\begin{eqnarray}\label{app-222}
 \sigma_{r}\bar{g}({T}, \mu)-\sigma_{r}^{\mu \mu} \nabla_{\mu} f&=&\sigma_{r}f\cot\theta-qf\sigma_{r}^{\mu \mu}\\
 &=&f\left(\cot \theta\left(\sigma_{r}-\sigma_{r}^{\mu \mu} h_{\mu \mu}\right)-\frac{\kappa}{\sin \theta} \sigma_{r}^{\mu \mu}\right) \nonumber\\
&=&f\left(\cot \theta \sigma_{r}\left(h \mid h_{\mu \mu}\right)-\frac{\kappa}{\sin \theta} \sigma_{r-1}\left(h \mid h_{\mu \mu}\right)\right).\nonumber
\end{eqnarray}
Here we used the fact
\begin{equation*}
  \sigma_{r}^{\mu\mu}=P^{\mu\mu}_{r-1}=\sigma_{r-1}(h|h_{\mu\mu}) \quad \text{and}\quad \sigma_{r}(h)=\sigma_{r}(h|h_{\mu\mu})+h_{\mu\mu}\sigma_{r-1}(h|h_{\mu\mu}).
\end{equation*}
By \eqref{nu0}, we see along $\partial M$
\begin{equation*}\label{xsssqaa}
 h_{\alpha \beta}=-\bar{g}(\bar{\nabla}_{e_{\alpha}}e_{\beta},\nu)=-\bar{g}(\bar{\nabla}_{e_{\alpha}}e_{\beta},\sin\theta\bar{\nu}-\cos\theta \bar{N})=\sin \theta \hat{h}_{\alpha \beta}-\kappa\cos \theta \delta_{\alpha \beta},
\end{equation*}
for an orthonormal frame $\left\{e_{\alpha}\right\}_{\alpha=1}^{n-1}$ of $T\left(\partial M\right)$. Thus
\begin{equation}\label{relations}
\sigma_{r}\left(h \mid h_{\mu \mu}\right)=\sigma_{r}\left(\sin \theta \hat{h}-\kappa\cos \theta I_{n-1}\right),
\end{equation}
where $I_{n-1}$ is the $(n-1)$-th identity matrix. In general, for a $(n-1) \times(n-1)$ symmetric matrix $B$, we know
\begin{equation*}
\sigma_{r}(I_{n-1}+B)=\sum_{l=0}^{r}\binom{n-l-1}{n-r-1}\sigma_{l}(B).
\end{equation*}

When $\kappa=0$. From \eqref{relations} we have
\begin{equation*}\label{app1233}
\cot \theta \sigma_{r}(h \mid h_{\mu \mu})=\cot\theta\sigma_{r}(\sin \theta \hat{h})=\cos\theta\sin^{r-1}\theta\sigma_{r}(\hat{h}).
\end{equation*}

When $\kappa\neq0$. From \eqref{relations} we get
\begin{eqnarray*}
\sigma_{r}\left(h \mid h_{\mu\mu}\right) &=&\sigma_{r}\left(\sin \theta \hat{h}-\kappa\cos \theta I_{n-1}\right) \\
&=&(-\kappa\cos\theta)^{r} \sum_{l=0}^{r}\binom{n-l-1}{n-r-1}\left(-\frac{\tan\theta}{\kappa}\right)^{l} \sigma_{l}(\hat{h}),
\end{eqnarray*}
and
\begin{eqnarray*}
\sigma_{r-1}\left(h \mid h_{\mu\mu}\right) &=&\sigma_{r-1}\left(\sin \theta \hat{h}-\kappa\cos \theta I_{n-1}\right) \\
&=&(-\kappa\cos \theta)^{r-1} \sum_{l=0}^{r-1}\binom{n-l-1}{n-r}\left(-\frac{\tan \theta}{\kappa}\right)^{l} \sigma_{l}(\hat{h}) .
\end{eqnarray*}
Therefore for any $\kappa\in\mathbb{R}$, we have
\begin{eqnarray}\label{app-var-111}
 &&\cot \theta \sigma_{r}\left(h \mid h_{\mu \mu}\right)-\frac{\kappa}{\sin \theta} \sigma_{r-1}\left(h \mid h_{\mu \mu}\right) \\
&=& \cos \theta \sin ^{r-1} \theta \sigma_{r}(\hat{h}) \nonumber\\
&&+\frac{\cos ^{r-1} \theta}{\sin \theta} \sum_{l=0}^{r-1}(-1)^{r+l}\kappa^{r-l}\left[\cos ^{2} \theta\binom{n-l-1}{n-r-1}+\binom{n-l-1}{n-r}\right] \tan ^{l} \theta \sigma_{l}(\hat{h}).\nonumber
\end{eqnarray}
Putting \eqref{app-var-111} and \eqref{app-222} into \eqref{var-formula} we see
\begin{eqnarray*}
\quad\frac{d}{dt} \int_{M} \sigma_{r}d A_{t}&=&(r+1) \int_{M} \sigma_{r+1}fdA_{t}-K(n-r+1)\int_{M}\sigma_{r-1}fdA_{t}+
\cos\theta \sin^{r-1}\theta \int_{\partial M} f \sigma_{r}(\hat{h})ds_{t}\\
&&+\frac{\cos^{r-1}\theta}{\sin\theta} \sum_{l=0}^{r-1}(-1)^{r+l}\kappa^{r-l}\left[\cos^{2}\theta\binom{n-l-1}{n-r-1}+\binom{n-l-1}{n-r}\right]\tan^{l}\theta \int_{\partial M} f \sigma_{l}(\hat{h})ds_{t}.
\end{eqnarray*}

Since $\hat{h}$ is the second fundamental form of $\partial M$ as a closed hypersurface in $\partial B$. Thus $\sigma_{r}(\hat{h})=\sigma^{\partial M}_{r}$ on $\partial M$. From \eqref{wetting-r}, we have the first variational formula of $r$-th wetting area functional $W_{r}$ as follows:
\begin{equation}
\frac{d}{d t} W_{r}(t)=\frac{1}{\sin\theta}{\binom{n}{r}}^{-1} \int_{\partial M} f \sigma_{r}(\hat{h})\,ds_{t}.
\end{equation}

We conclude that for $1 \leq r \leq n-1$,
\begin{eqnarray}\label{w1}
&&\frac{d}{d t}\left\{\int_{M}\sigma_{r} d A_{t}-\binom{n}{r}\cos \theta \sin^{r} \theta W_{r}(t)\right. \\
&&\left.\quad-\cos ^{r-1} \theta \sum_{l=0}^{r-1}(-1)^{r+l}\kappa^{r-l}\binom{n}{l}\left[\cos^{2}\theta\binom{n-l-1}{n-r-1}+\binom{n-l-1}{n-r}\right] \tan ^{l} \theta W_{l}(t)\right\}\nonumber \\
&=&(r+1) \int_{M} \sigma_{r+1} f dA_{t}-K(n-r+1)\int_{M}\sigma_{r-1}f dA_{t}.\nonumber
\end{eqnarray}
By the combination relationships
\begin{equation*}
  \binom{n}{l}\cdot\binom{n-l-1}{n-r-1}=\binom{n}{r}\cdot\binom{r}{l}\frac{n-r}{n-l}
\end{equation*}
and
\begin{equation*}
  \binom{n}{l}\cdot\binom{n-l-1}{n-r}=\binom{n}{r}\cdot\binom{r}{l}\frac{r-l}{n-l}.
\end{equation*}
From \eqref{w1}, we obtain
\begin{equation}\label{w2}
\begin{aligned}
&\frac{d}{d t}\left\{\int_{M}\sigma_{r} d A_{t}-\binom{n}{r}\cos \theta \sin^{r} \theta W_{r}(t)\right. \\
&\left.\quad-\cos ^{r-1} \theta \sum_{l=0}^{r-1}\frac{(-1)^{r+l}\kappa^{r-l}}{n-l}\binom{r}{l}\left[(n-r)\cos^{2}\theta+(r-l)\right]\tan ^{l} \theta \binom{n}{r} W_{l}(t)\right\} \\
&=(r+1) \int_{M} \sigma_{r+1} f dA_{t}-K(n-r+1)\int_{M}\sigma_{r-1}f dA_{t}.
\end{aligned}
\end{equation}
Recall that
$$
\begin{aligned}
&Q_{r+1}(t)=\int_{M}H_{r} d A_{t}-\cos \theta \sin^{r} \theta W_{r}(t)\\
&\qquad\qquad-\cos ^{r-1} \theta \sum_{l=0}^{r-1}\frac{(-1)^{r+l}\kappa^{r-l}}{n-l}\binom{r}{l}\left[(n-r)\cos^{2}\theta+(r-l)\right] \tan^{l}\theta W_{l}(t).
\end{aligned}
$$
Therefore, by \eqref{w2} we have
\begin{equation}\label{w34}
\frac{d}{d t}Q_{r+1}(t)=(n-r)\int_{M} H_{r+1} f dA_{t}-rK\int_{M}H_{r-1}f dA_{t}.
\end{equation}
Let
\begin{eqnarray}\label{w35}
\mathcal{E}_{r+1}(t)=Q_{r+1}(t)+\frac{rK}{n+2-r}{\mathcal{E}}_{r-1}(t).
\end{eqnarray}
One can readily check that for any $-1\leq s\leq n-1$
\begin{equation}\label{w3}
\frac{d}{dt}{\mathcal{E}}_{s+1}(t)=(n-s)\int_{M_{t}} H_{s+1} f dA.
\end{equation}
In fact, it is true for $s=-1$ and $s=0$. By induction, we assume it is true for $s=r-2$, we can calculate by using \eqref{w34} and \eqref{w35}
\begin{eqnarray*}
&&\frac{d}{dt}{\mathcal{E}}_{r+1}(t)\\
&=&\frac{d}{dt}Q_{r+1}(t)+\left(\frac{rK}{n+2-r}\right)\frac{d}{dt}{\mathcal{E}}_{r-1}(t)\\
&=&(n-r)\int_{M} H_{r+1} f dA_{t}-rK\int_{M}H_{r-1}f dA_{t}+\frac{rK}{n+2-r}(n+2-r)\int_{M}H_{r-1}f dA_{t}\\
&=&(n-r)\int_{M} H_{r+1} f dA_{t}.
\end{eqnarray*}
We complete the proof of Theorem \ref{thm105-half}.

\qed

\

\textbf{Acknowledgments.} The authors would like to thank Professor Guofang Wang for his constant support and his interest on this topic.

\

\textbf{Funding.} {J. Guo was supported by China Postdoctoral Science Foundation (No.2022M720079) and Shuimu Tsinghua Scholar Program (No.2022SM046). H. Li was supported by NSFC (Grant No.12471047). C. Xia was supported by NSFC (Grant No.12271449, 12126102) and the Natural Science Foundation of Fujian Province of China (Grant No.2024J011008).}

\

\end{document}